\documentclass[letterpaper,11pt]{article}

\usepackage{fullpage}		
\usepackage{amsfonts}
\usepackage{mathtools}	

\usepackage{amssymb}
\usepackage{amsmath}
\usepackage{amsthm}
\usepackage{color}
\usepackage{kbordermatrix}
\usepackage{graphicx}

\usepackage{relsize}

\usepackage{float}		

\usepackage[colorlinks=true, citecolor=blue, urlcolor=cyan, linktocpage]{hyperref}

\usepackage{enumerate}	

\usepackage{tikz}				
\usetikzlibrary{arrows,matrix,positioning}

\usepackage{wrapfig}

\usepackage{cancel} 		

\usepackage{verbatim}   

\usepackage[toc]{appendix}	

\usepackage{soul} 

\newtheorem{theorem}{Theorem}[section] 

\newtheorem{corollary}[theorem]{Corollary}
\newtheorem{conjecture}[theorem]{Conjecture}

\newtheorem{proposition}[theorem]{Proposition}

\theoremstyle{definition}
\newtheorem{example}[theorem]{Example}

\theoremstyle{definition}
\newtheorem{definition}[theorem]{Definition}

\theoremstyle{plain}

\theoremstyle{definition}
\newtheorem{remark}[theorem]{Remark}

\theoremstyle{definition}
\newtheorem{question}[theorem]{Question}

\DeclareMathOperator{\spec}{\texttt{spec}}

\renewcommand{\Re}{\operatorname{Re}}
\renewcommand{\Im}{\operatorname{Im}}

\numberwithin{figure}{section}
\numberwithin{table}{section}
\numberwithin{theorem}{section}
\numberwithin{equation}{section}

\date{}

\title{\textcolor{red}{\textbf{Tridiagonal real symmetric matrices with a connection to Pascal's triangle and the Fibonacci sequence}}}

\author{
	Emily Gullerud\\
  \href{mailto:gulle069@umn.edu}{\nolinkurl{gulle069@umn.edu}}
  \and
	Rita Johnson\\
  \href{mailto:rita.laraine.johnson@gmail.com}{\nolinkurl{rita.laraine.johnson@gmail.com}}
  \and
  aBa Mbirika\thanks{Corresponding author}\\
  \href{mailto:mbirika@uwec.edu}{\nolinkurl{mbirika@uwec.edu}}
}

\newcommand{\nocontentsline}[3]{}
\newcommand{\tocless}[2]{\bgroup\let\addcontentsline=\nocontentsline#1{#2}\egroup}

\makeatletter
\newcommand\ackname{Acknowledgments}
\if@titlepage
  \newenvironment{acknowledgments}{%
      \titlepage
      \null\vfil
      \@beginparpenalty\@lowpenalty
      \begin{center}%
        \bfseries \ackname
        \@endparpenalty\@M
      \end{center}}%
     {\par\vfil\null\endtitlepage}
\else
  
\fi
\makeatother

\makeatletter
\newcommand{\subjclass}[2][2020]{%
  \let\@oldtitle\@title%
  \gdef\@title{\@oldtitle\footnotetext{#1 \emph{Mathematics subject classification.} #2}}%
}
\newcommand{\keywords}[1]{%
  \let\@@oldtitle\@title%
  \gdef\@title{\@@oldtitle\footnotetext{\emph{Keywords.} #1.}}%
}
\makeatother

\keywords{tridiagonal matrix, characteristic polynomial, eigenvalue, Pascal's triangle, Fibonacci sequence}
\subjclass{Primary 05C50, 15A18; Secondary 65F15}

\begin{document}

\maketitle

\begin{abstract}
We explore a certain family $\{A_n\}_{n=1}^{\infty}$ of $n \times n$ tridiagonal real symmetric matrices. After deriving a three-term recurrence relation for the characteristic polynomials of this family, we find a closed form solution. The coefficients of these characteristic polynomials turn out to involve the diagonal entries of Pascal's triangle in a tantalizingly predictive manner. Lastly, we explore a relation between the eigenvalues of various members of the family. More specifically, we give a sufficient condition on the values $m,n \in \mathbb{N}$ for when $\spec(A_m)$ is contained in $\spec(A_n)$. We end the paper with a number of open questions, one of which intertwines our characteristic polynomials with the Fibonacci sequence in an intriguing manner involving ellipses.
\end{abstract}


\tableofcontents 


\section{Introduction}
Tridiagonal matrices arise in many science and engineering areas, for example in parallel computing, telecommunication system analysis, and in solving differential equations using finite differences~\cite{Bar-On1996, LuSun1997, Crybaby1972}. In particular, the eigenvalues of tridiagonal \textit{symmetric} matrices have been studied extensively starting with Golub in 1962~\cite{Golub1962}. Moreover, a search on MathSciNet reveals that over 100 papers with the words ``tridiagonal symmetric matrices'' in the title have been published since then.

Tridiagonal real symmetric matrices are a subclass of the class of real symmetric matrices. A \textit{real symmetric matrix} is a square matrix $A$ with real-valued entries that are symmetric about the diagonal; that is, $A$ equals its transpose $A^\mathrm{T}$. Symmetric matrices arise naturally in a variety of applications. Real symmetric matrices in particular enjoy the following two properties: (1) all of their eigenvalues are real, and (2) the eigenvectors corresponding to distinct eigenvalues are orthogonal. In this paper, we investigate the family of real symmetric $n \times n$ matrices of the form
$$
A_n = \begin{bmatrix} 
0 & 1 &  & \\
1 & \ddots & \ddots & \\
 & \ddots & \ddots & 1\\
 &  & 1 & 0
\end{bmatrix}
$$
with ones on the superdiagonal and subdiagonal and zeroes in every other entry. These matrices, in particular, arise as the adjacency matrices of the path graphs and hence are fundamental objects in the field of spectral graph theory. The three main results of this paper are
\begin{itemize}
\item \textbf{MAIN RESULT 1}: We give a closed form expression for the characteristic polynomials $f_n(\lambda)$ of the $A_n$ matrices and establish a three-term recurrence relation which the $f_n(\lambda)$ polynomials satisfy.
\item \textbf{MAIN RESULT 2}: We show that the coefficients of these $f_n(\lambda)$ polynomials have an intimate connection to the diagonals of Pascal's triangle.
\item \textbf{MAIN RESULT 3}: Utilizing a trigonometric closed form expression for the set of eigenvalues of the $A_n$ matrices, we give an upper bound on the spectral radius $\rho(A_n)$ of $A_n$ and provide a sufficiency condition on the values $m,n \in \mathbb{N}$ that guarantees $\spec(A_m) \subset \spec(A_n)$.
\end{itemize}

We were initially drawn to this work when noticing the first few characteristic polynomials of the $A_n$ matrices. For example, we give the first seven polynomials below.
$$
\begin{array}{|c||c c c c c|}
\hline
 n &&& \hspace{-.5in}f_n(\lambda) &&\\
\hline 
\hline
 1 & - \textcolor{cyan}{1}\lambda &&&&\\
 2 & \textcolor{cyan}{1}\lambda ^2 & - \textcolor{orange}{1} &&&\\
 3 & - \textcolor{cyan}{1}\lambda ^3 & + \textcolor{orange}{2} \lambda  &&&\\
 4 & \textcolor{cyan}{1}\lambda ^4 & - \textcolor{orange}{3} \lambda ^2 & + \textcolor{violet}{1} &&\\
 5 & - \textcolor{cyan}{1}\lambda ^5 & + \textcolor{orange}{4} \lambda ^3 & - \textcolor{violet}{3} \lambda &&\\
 6 & \textcolor{cyan}{1}\lambda ^6 & - \textcolor{orange}{5} \lambda ^4 & + \textcolor{violet}{6} \lambda ^2 & - \textcolor{teal}{1} &\\
 7 & - \textcolor{cyan}{1}\lambda ^7 & + \textcolor{orange}{6} \lambda ^5 & - \textcolor{violet}{10} \lambda ^3 & + \textcolor{teal}{4} \lambda  &\\
\hline
\end{array}
$$
Color-coding the coefficients to make the connection more transparent, we see that the four columns of coefficients (up to absolute value) correspond directly to the the first four diagonals in Pascal's triangle. We used this observation to conjecture and eventually prove the following formula for the $m^{\text{th}}$ characteristic polynomial:
$$f_m(\lambda) = \sum_{i=0}^{\left \lfloor{\frac{m}{2}}\right\rfloor } (-1)^{m+i} \binom{m-i}{i} \lambda^{m-2i}.$$
We prove that this formula satisfies the three-term recurrence relation
$$f_n(\lambda) = -\lambda f_{n-1}(\lambda) - f_{n-2}(\lambda)$$
with initial conditions $f_1(\lambda) = -\lambda$ and $f_2(\lambda) = \lambda^2 - 1$, thereby establishing our first main result. Our second main result uses a closed form for the eigenvalues of the matrix $A_n$ to explore some properties of the spectrum of $A_n$. Our third main result was motivated by an observation that the four distinct eigenvalues of the matrix $A_4$ are the golden ratio, its reciprocal, and their additive inverses. Furthermore via \texttt{Mathematica}, we observed that these four eigenvalues appear in the spectrum of the $A_n$ matrices for the $n$-values 4, 9, 14, 19, 24, 29, 34, 39, and 44. These numbers all being congruent to 4 modulo 5 seemed to be no accident. Motivated by this tantalizing observation, we use a closed form expression
$$\lambda_s = 2\cos\left( \frac{s \pi}{m+1} \right) \text{ for } s=1,\ldots,n$$
for the $n$ distinct eigenvalues of each $A_n$ matrix to give a sufficiency criterion for when the eigenvalues of the matrix $A_m$ are also eigenvalues of the matrix $A_n$, thereby establishing our third main result.

The breakdown of the paper is as follows. In Section~\ref{sec:prelims}, we give some preliminaries and definitions. In Section~\ref{sec:char_polys}, we focus on the characteristic polynomials of the $A_n$ matrices; in particular,
\begin{enumerate}[\hspace{.5cm}(1)]
\item we derive a three-term recurrence relation for the characteristic polynomials of the family of matrices $\{A_n\}_{n=1}^\infty$ in Subsection~\ref{subsec:recurrence_relation},
\item we unveil a tantalizing connection between the coefficients of the family of characteristic polynomials $\{f_n(\lambda)\}_{n=1}^\infty$ of the matrices $\{A_n\}_{n=1}^\infty$ and the diagonal columns of Pascal's triangle in Subsection~\ref{subsec:char_poly_and_Pascal},
\item we provide a closed form expression for $f_n(\lambda)$ and prove that this closed form satisfies the recurrence relation in Subsection~\ref{subsec:closed_form_of_char_poly},
\item we give a parity property of each $f_n(\lambda)$ dependent on the parity of $n$ in Subsection~\ref{subsec:some_properties}, and
\item we give a connection between our family of characteristic polynomials $\{f_n(\lambda)\}_{n=1}^\infty$ and Chebyshev polynomials of the second kind in Subsection~\ref{subsec:Chebyshev}.
\end{enumerate}
In Section~\ref{sec:spectrum}, we focus on the spectrum of the $A_n$ matrices; in particular,
\begin{enumerate}[\hspace{.5cm}(1)]
\setcounter{enumi}{5}
\item we prove that a given trigonometric closed form yields the eigenvalues of each matrix $A_n$ in Subsection~\ref{subsec:closed_form_for_roots},
\item we investigate the spectral radius of the matrix $A_n$ and in particular give lower and upper bounds of $\spec(A_n)$ in Subsection~\ref{subsec:spectral_radius}, and
\item we prove a sufficiency condition on the values $m,n \in \mathbb{N}$ that guarantees the containment $\spec(A_m) \subset \spec(A_n)$ in Subsection~\ref{subsec:suff_cond}.
\end{enumerate}
Finally in Section~\ref{sec:open_questions}, we provide some open questions and, in particular, explore an intriguing connection between the $f_n(\lambda)$ polynomials and the Fibonacci sequence.


\section{Preliminaries and definitions}\label{sec:prelims}

Let us recall some fundamental linear algebra definitions used in this paper. Since our main results revolve around eigenvalues, we start with this definition and establish notation.
\begin{definition}[Eigenvalue and eigenvector]
Let $A \in \mathbb{C}^{n \times n}$ be a square matrix. An \textit{eigenpair} of $A$ is a pair $(\lambda, \vec{v}) \in \mathbb{C} \times (\mathbb{C}^n-{\vec{0}})$ such that $A \vec{v} = \lambda \vec{v}$. We call $\lambda$ an \textit{eigenvalue} and its corresponding nonzero vector $\vec{v}$ an \textit{eigenvector}.
\end{definition}
In particular, we look at the set of eigenvalues of a given matrix $A$ and the largest element up to absolute value of these eigenvalues defined as follows.
\begin{definition}[Spectrum and spectral radius]
Let $A \in \mathbb{C}^{n \times n}$ be a square matrix. The multiset of eigenvalues of $A$ is called the \textit{spectrum} of $A$ and is denoted $\spec(A)$. The \textit{spectral radius} of $A$ is denoted $\rho(A)$ and defined to be
$$\rho(A) = \max\{|\lambda| : \lambda \in \spec(A) \}.$$
\end{definition}

The matrices we care about in this paper are a subset of a wider class of symmetric matrices called Toeplitz matrices. A \textit{Toeplitz matrix} is a matrix of the form
$$
\begin{bmatrix} 
  a_{0} & a_{1} & a_{2} & \ldots & \ldots &a_{n-1} \\
  a_{1} & a_0 & a_{1} & \ddots & & \vdots \\
  a_{2} & a_{1} & \ddots & \ddots & \ddots& \vdots \\
  \vdots & \ddots & \ddots & \ddots & a_{1} & a_{2}\\
  \vdots & & \ddots & a_{1} & a_{0} & a_{1} \\
  a_{n-1} & \ldots & \ldots & a_{2} & a_{1} &a_{0} 
\end{bmatrix}.
$$
In particular, we focus on a subclass of Toeplitz matrices called tridiagonal symmetric matrices. We give the specific definitions below.
\begin{definition}[Tridiagonal matrix]
A \textit{tridiagonal symmetric matrix} is a Toeplitz matrix in which all entries not lying on the diagonal, superdiagonal, or subdiagonal are zero.
\end{definition}
We limit our perspective by considering the tridiagonal matrices of the following form.
\begin{definition}[The matrix $A_n$]
Let $A_n$ be an $n \times n$ tridiagonal symmetric matrix in which the diagonal entries are zero, and the superdiagonal and subdiagonal entries are all one.
\end{definition}

\begin{example}
Here are the $A_n$ matrices for $n=1,\ldots,4$.
$$
A_1 = 
\begin{bmatrix} 
0
\end{bmatrix}
\hspace{.35in}
A_2 = 
\begin{bmatrix} 
0 & 1\\
1 & 0
\end{bmatrix}
\hspace{.35in}
A_3 = 
\begin{bmatrix} 
0 & 1 & 0\\
1 & 0 & 1\\
0 & 1 & 0
\end{bmatrix}
\hspace{.35in}
A_4 = 
\begin{bmatrix} 
0 & 1 & 0 & 0\\
1 & 0 & 1 & 0\\
0 & 1 & 0 & 1\\
0 & 0 & 1 & 0
\end{bmatrix}
$$
\end{example}
A typical method of finding the eigenvalues of a square matrix is by calculating the roots of the \textit{characteristic polynomial} of the matrix. 
\begin{definition}[The characteristic polynomial $f_n(\lambda)$]
Given the matrix $A_n$, the \textit{characteristic polynomial} $f_n(\lambda)$ is the determinant of the matrix $A_n - \lambda I_n$; that is, $f_n(\lambda) = |A_n - \lambda I_n|$.
\end{definition}

\begin{remark}
If we set $f_n(\lambda) = 0$, then clearly the $n$ roots (i.e., eigenvalues) of this characteristic equation give the spectrum of $A_n$. Moreover, these eigenvalues are real since $A_n$ is a real symmetric matrix~\cite[Fact~1-2]{Parlett1980}, and these eigenvalues are distinct since the subdiagonal and superdiagonal entries of $A_n$ are nonzero~\cite[Lemma~7-7-1]{Parlett1980}. In particular, $A_n$ has $n$ distinct real eigenvalues.
\end{remark}

\begin{example}\label{exam:golden_ratio}
Consider the matrix $A_4$ and the determinant $|A_4 - \lambda I_4|$, which gives the characteristic polynomial $f_4(\lambda)$.
$$
A_4 = 
\begin{bmatrix} 
0 & 1 & 0 & 0\\
1 & 0 & 1 & 0\\
0 & 1 & 0 & 1\\
0 & 0 & 1 & 0
\end{bmatrix}
\hspace{.65in}
|A_4 - \lambda I_4| = \begin{vmatrix} 
-\lambda & 1 & 0 & 0\\
1 & -\lambda & 1 & 0\\
0 & 1 & -\lambda & 1\\
0 & 0 & 1 & -\lambda
\end{vmatrix} = \lambda^4 - 3 \lambda^2 + 1.
$$
Thus the characteristic polynomial is $f_4(\lambda) = \lambda^4 - 3 \lambda^2 + 1$. The roots of the corresponding characteristic equation $f_4(\lambda) = 0$ yield the four distinct eigenvalues
\begin{align*}
\lambda_1 &=\frac{1+\sqrt{5}}{2} = \phi \approx 1.61803 & \lambda_3 &=\frac{-1+\sqrt{5}}{2} = \frac{1}{\phi} \approx .61803\\
\lambda_2 &=\frac{1-\sqrt{5}}{2} = -\frac{1}{\phi} \approx -.61803 & \lambda_4 &=\frac{-1-\sqrt{5}}{2} = -\phi \approx -1.61803,
\end{align*}
where $\phi$ is the golden ratio. Moreover in Subsection~\ref{subsec:suff_cond}, we determine an infinite set of $n$ values for which $\spec(A_4) \subset \spec(A_n)$.
\end{example}


\section{The characteristic polynomial \texorpdfstring{$f_n(\lambda)$}{} of the matrix \texorpdfstring{$A_n$}{the A sub n matrix}}\label{sec:char_polys}
In this section, we find the family of characteristic polynomials $\{f_n(\lambda)\}_{n=1}^{\infty}$ corresponding to the family of matrices $\{A_n\}_{n=1}^\infty$. For each $n \in \mathbb{N}$, the coefficients of the polynomial $f_n(\lambda)$ reveal themselves to be exactly the entries of the $(n+1)^{\mathrm{th}}$ diagonal of Pascal's triangle as denoted in Figure~\ref{fig:pascalTri}. We use this observation to find a closed form expression for $f_n(\lambda)$ for each $n$. To confirm this expression is indeed correct, we show that it satisfies the easily derived three-term recurrence relation that the sequence $\{f_n(\lambda)\}_{n=1}^{\infty}$ must follow.

\subsection{A recurrence relation for \texorpdfstring{$f_n(\lambda)$}{the characteristic polynomial}}\label{subsec:recurrence_relation}

In this subsection, we show that the determinants of the matrices $A_n - \lambda I_n$, that is, the characteristic polynomials $f_n(\lambda)$, satisfy the following three-term recurrence relation for all $n \geq 3$:
\begin{equation}\label{eqn:recurrence_relation}
f_n(\lambda) = -\lambda f_{n-1}(\lambda) - f_{n-2}(\lambda)
\end{equation}
with initial conditions $f_1(\lambda) = -\lambda$ and $f_2(\lambda) = \lambda^2 -1$. First we consider the $n \times n$ matrix $A_n-\lambda I_n$ and note that bottom-right $(n-1) \times (n-1)$ submatrix is $A_{n-1} -\lambda I_{n-1}$, and hence the minor given by deletion of the first row and first column is the characteristic polynomial $f_{n-1}(\lambda)$. This fact is pivotal in deriving the recurrence relation $f_n(\lambda)$.
\newline

\begin{tikzpicture}
\hspace{-2cm}
$A_n-\lambda I_n=$
\hspace{2.7cm}
        \matrix [matrix of math nodes,left delimiter={[},right delimiter={]}] (m)
        {
            -\lambda & 1 & 0 & \cdots & 0\\               
            1 & -\lambda & 1 & \ddots & 0 \\         
            0 & 1 & -\lambda & \ddots & \vdots\\
						\vdots & \ddots & \ddots & \ddots & 1\\
						0 & 0 & \cdots & 1 & -\lambda\\						
        };    

 \hspace{2.7cm}$=$
\hspace{2.85cm}
        \matrix [matrix of math nodes,left delimiter={[},right delimiter={]}] (m)
        {
            -\lambda &1 & 0 & \ldots & 0\\               
            1 & \phantom{x} &\phantom{x} & \phantom{x} & \phantom{x} \\         
            0 & \phantom{x} & A_{n-1} &\hspace{-.1in}-\lambda I_{n-1} & \phantom{x}\\
						\vdots & \phantom{x}& \phantom{x} &\phantom{x} & \phantom{x}\\
						0 & \phantom{x} & \phantom{x} & \phantom{x} & \phantom{x}\\						
        };    
			\draw[color=red] (m-2-2.north west) -- (m-2-5.north east) -- (m-5-5.south east) -- (m-5-2.south west) -- (m-2-2.north west);
\end{tikzpicture}

\noindent Computing the determinant by cofactor expansion along the top row of $A_n - \lambda I_n$, we derive the following.	

\begingroup
\allowdisplaybreaks
\begin{align*}
f_n(\lambda)&= |A_n - \lambda I_n| \\
&=-\lambda \cdot |A_{n-1}-\lambda I_{n-1}|-1\cdot
\begin{vmatrix}
						1 & 1 & 0  &\cdots & 0\\
						0 & -\lambda & 1 & \ddots & 0\\               
            0 & 1 & -\lambda & \ddots & \vdots \\         
						\vdots & \ddots & \ddots & \ddots & 1\\
						0 & 0 & \cdots & 1 & -\lambda\\		
\end{vmatrix}\\
&=-\lambda \cdot f_{n-1}(\lambda)-1\cdot
\begin{tikzpicture}[baseline=-0.5ex]
 \matrix [matrix of math nodes,left delimiter=|,right delimiter=|,ampersand replacement=\&] (m)
        {
            1 \& 1 \& 0 \& \ldots \& 0\\               
            0 \& \phantom{x} \&\phantom{x} \& \phantom{x} \& \phantom{x} \\         
            \vdots \& \phantom{x} \& A_{n-2} \&\hspace{-.1in}-\lambda I_{n-2} \& \phantom{x}\\
						\vdots \& \phantom{x} \& \phantom{x} \& \phantom{x} \& \phantom{x}\\
						0 \& \phantom{x} \& \phantom{x} \& \phantom{x} \& \phantom{x}\\						
        };     
			\draw[color=red] (m-2-2.north west) -- (m-2-5.north east) -- (m-5-5.south east) -- (m-5-2.south west) -- (m-2-2.north west);
\end{tikzpicture}\\
&=-\lambda \cdot f_{n-1}(\lambda) - f_{n-2}(\lambda).
\end{align*}%
\endgroup

\subsection{The \texorpdfstring{$f_n(\lambda)$}{characteristic polynomial} and their connection to Pascal's triangle}\label{subsec:char_poly_and_Pascal}

In Table~\ref{table:char_polys}, we give the characteristic polynomials $f_n(\lambda)$ for $n=1,\ldots,12$. Upon examination of the coefficients of these twelve polynomials, it becomes immediately apparent that the binomial coefficients in the diagonal columns of Pascal's triangle are intimately involved with the coefficients of the $f_n(\lambda)$. In Figure~\ref{fig:pascalTri}, we give the first twelve rows of Pascal's triangle, carefully coding the diagonal columns, to match the table of the twelve $f_n(\lambda)$ polynomials with their corresponding colors matching those in Pascal's triangle.

\begin{table}[h!]
$$
\begin{array}{|c||c c c c c c c|}
\hline
 n &&& f_n(\lambda) &&&&\\
\hline 
\hline
 1 & - \textcolor{cyan}{1}\lambda &&&&&&\\
 2 & \textcolor{cyan}{1}\lambda ^2 & - \textcolor{orange}{1} &&&&&\\
 3 & - \textcolor{cyan}{1}\lambda ^3 & + \textcolor{orange}{2} \lambda  &&&&&\\
 4 & \textcolor{cyan}{1}\lambda ^4 & - \textcolor{orange}{3} \lambda ^2 & + \textcolor{violet}{1} &&&&\\
 5 & - \textcolor{cyan}{1}\lambda ^5 & + \textcolor{orange}{4} \lambda ^3 & - \textcolor{violet}{3} \lambda &&&&\\
 6 & \textcolor{cyan}{1}\lambda ^6 & - \textcolor{orange}{5} \lambda ^4 & + \textcolor{violet}{6} \lambda ^2 & - \textcolor{teal}{1} &&&\\
 7 & - \textcolor{cyan}{1}\lambda ^7 & + \textcolor{orange}{6} \lambda ^5 & - \textcolor{violet}{10} \lambda ^3 & + \textcolor{teal}{4} \lambda  &&&\\
 8 & \textcolor{cyan}{1}\lambda ^8 & - \textcolor{orange}{7} \lambda ^6 & + \textcolor{violet}{15} \lambda ^4 & - \textcolor{teal}{10} \lambda ^2 & + \textcolor{magenta}{1} &&\\
 9 & - \textcolor{cyan}{1}\lambda ^9 & + \textcolor{orange}{8} \lambda ^7 & - \textcolor{violet}{21} \lambda ^5 & + \textcolor{teal}{20} \lambda ^3 & - \textcolor{magenta}{5} \lambda &&\\
 10 & \textcolor{cyan}{1}\lambda ^{10} & - \textcolor{orange}{9} \lambda ^8 & + \textcolor{violet}{28} \lambda ^6 & - \textcolor{teal}{35} \lambda ^4 & + \textcolor{magenta}{15} \lambda ^2 & - \textcolor{green}{1} &\\
 11 & - \textcolor{cyan}{1}\lambda ^{11} & + \textcolor{orange}{10} \lambda ^9 & - \textcolor{violet}{36} \lambda ^7 & + \textcolor{teal}{56} \lambda ^5 & - \textcolor{magenta}{35} \lambda ^3 & + \textcolor{green}{6} \lambda &\\
 12 & \textcolor{cyan}{1}\lambda ^{12} & - \textcolor{orange}{11} \lambda ^{10} & + \textcolor{violet}{45} \lambda ^8 & - \textcolor{teal}{84} \lambda ^6 & + \textcolor{magenta}{70} \lambda ^4 & - \textcolor{green}{21} \lambda ^2 & + \textcolor{blue}{1}\\
\hline
\end{array}
$$
\caption{The first twelve characteristic polynomials $f_n(\lambda)$}
\label{table:char_polys}
\end{table}

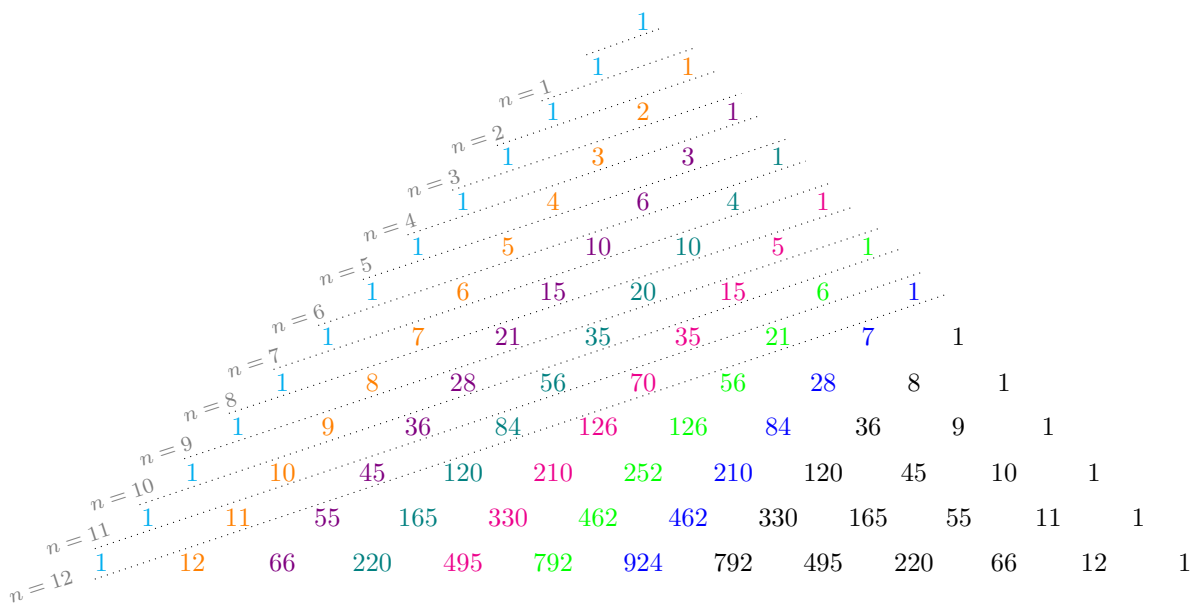
\begin{figure}[h!]
\centering
\begin{tikzpicture}[scale=.9]
\node at (0,8) {\small{\textcolor{cyan}{1}}};
\node at (-0.665,7.33) {\small{\textcolor{cyan}{1}}};
\node at (0.665,7.33) {\small{\textcolor{orange}{1}}};
\node at (-1.33,6.66) {\small{\textcolor{cyan}{1}}};
\node at (0,6.66) {\small{\textcolor{orange}{2}}};
\node at (1.33,6.66) {\small{\textcolor{violet}{1}}};
\node at (-1.995,6) {\small{\textcolor{cyan}{1}}};
\node at (-0.665,6) {\small{\textcolor{orange}{3}}};
\node at (0.665,6) {\small{\textcolor{violet}{3}}};
\node at (1.995,6) {\small{\textcolor{teal}{1}}};
\node at (-2.66,5.33) {\small{\textcolor{cyan}{1}}};
\node at (-1.33,5.33) {\small{\textcolor{orange}{4}}};
\node at (0,5.33) {\small{\textcolor{violet}{6}}};
\node at (1.33,5.33) {\small{\textcolor{teal}{4}}};
\node at (2.66,5.33) {\small{\textcolor{magenta}{1}}};
\node at (-3.325,4.66) {\small{\textcolor{cyan}{1}}};
\node at (-1.995,4.66) {\small{\textcolor{orange}{5}}};
\node at (-0.665,4.66) {\small{\textcolor{violet}{10}}};
\node at (0.665,4.66) {\small{\textcolor{teal}{10}}};
\node at (1.995,4.66) {\small{\textcolor{magenta}{5}}};
\node at (3.325,4.66) {\small{\textcolor{green}{1}}};
\node at (-4,4) {\small{\textcolor{cyan}{1}}};
\node at (-2.66,4) {\small{\textcolor{orange}{6}}};
\node at (-1.33,4) {\small{\textcolor{violet}{15}}};
\node at (0,4) {\small{\textcolor{teal}{20}}};
\node at (1.33,4) {\small{\textcolor{magenta}{15}}};
\node at (2.66,4) {\small{\textcolor{green}{6}}};
\node at (4,4) {\small{\textcolor{blue}{1}}};
\node at (-4.655,3.33) {\small{\textcolor{cyan}{1}}};
\node at (-3.325,3.33) {\small{\textcolor{orange}{7}}};
\node at (-1.995,3.33) {\small{\textcolor{violet}{21}}};
\node at (-0.665,3.33) {\small{\textcolor{teal}{35}}};
\node at (0.665,3.33) {\small{\textcolor{magenta}{35}}};
\node at (1.995,3.33) {\small{\textcolor{green}{21}}};
\node at (3.325,3.33) {\small{\textcolor{blue}{7}}};
\node at (4.655,3.33) {\small 1};
\node at (-5.33,2.66) {\small{\textcolor{cyan}{1}}};
\node at (-4,2.66) {\small{\textcolor{orange}{8}}};
\node at (-2.66,2.66) {\small{\textcolor{violet}{28}}};
\node at (-1.33,2.66) {\small{\textcolor{teal}{56}}};
\node at (0,2.66) {\small{\textcolor{magenta}{70}}};
\node at (1.33,2.66) {\small{\textcolor{green}{56}}};
\node at (2.66,2.66) {\small{\textcolor{blue}{28}}};
\node at (4,2.66) {\small 8};
\node at (5.33,2.66) {\small 1};
\node at (-5.985,2) {\small{\textcolor{cyan}{1}}};
\node at (-4.655,2) {\small{\textcolor{orange}{9}}};
\node at (-3.325,2) {\small{\textcolor{violet}{36}}};
\node at (-1.995,2) {\small{\textcolor{teal}{84}}};
\node at (-0.665,2) {\small{\textcolor{magenta}{126}}};
\node at (0.665,2) {\small{\textcolor{green}{126}}};
\node at (1.995,2) {\small{\textcolor{blue}{84}}};
\node at (3.325,2) {\small 36};
\node at (4.655,2) {\small 9};
\node at (5.985,2) {\small 1};
\node at (-6.66,1.33) {\small{\textcolor{cyan}{1}}};
\node at (-5.33,1.33) {\small{\textcolor{orange}{10}}};
\node at (-4,1.33) {\small{\textcolor{violet}{45}}};
\node at (-2.66,1.33) {\small{\textcolor{teal}{120}}};
\node at (-1.33,1.33) {\small{\textcolor{magenta}{210}}};
\node at (0,1.33) {\small{\textcolor{green}{252}}};
\node at (1.33,1.33) {\small{\textcolor{blue}{210}}};
\node at (2.66,1.33) {\small 120};
\node at (4,1.33) {\small 45};
\node at (5.33,1.33) {\small 10};
\node at (6.66,1.33) {\small 1};
\node at (-7.315,0.66) {\small{\textcolor{cyan}{1}}};
\node at (-5.985,0.66) {\small{\textcolor{orange}{11}}};
\node at (-4.665,0.66) {\small{\textcolor{violet}{55}}};
\node at (-3.325,0.66) {\small{\textcolor{teal}{165}}};
\node at (-1.995,0.66) {\small{\textcolor{magenta}{330}}};
\node at (-0.665,0.66) {\small{\textcolor{green}{462}}};
\node at (0.665,0.66) {\small{\textcolor{blue}{462}}};
\node at (1.995,0.66) {\small 330};
\node at (3.325,0.66) {\small 165};
\node at (4.665,0.66) {\small 55};
\node at (5.985,0.66) {\small 11};
\node at (7.315,0.66) {\small 1};
\node at (-8,0) {\small{\textcolor{cyan}{1}}};
\node at (-6.66,0) {\small{\textcolor{orange}{12}}};
\node at (-5.33,0) {\small{\textcolor{violet}{66}}};
\node at (-4,0) {\small{\textcolor{teal}{220}}};
\node at (-2.66,0) {\small{\textcolor{magenta}{495}}};
\node at (-1.33,0) {\small{\textcolor{green}{792}}};
\node at (0,0) {\small{\textcolor{blue}{924}}};
\node at (1.33,0) {\small 792};
\node at (2.66,0) {\small 495};
\node at (4,0) {\small 220};
\node at (5.33,0) {\small 66};
\node at (6.66,0) {\small 12};
\node at (8,0) {\small 1};
\draw[dotted] (-8.1,-0.25) -- (4.45,3.95);
\draw[dotted] (-8.1,0.2) -- (4.1,4.28);
\draw[dotted] (-7.44,0.85) -- (3.8,4.63);
\draw[dotted] (-6.78,1.53) -- (3.47,4.93);
\draw[dotted] (-6.12,2.2) -- (3.1,5.25);
\draw[dotted] (-5.46,2.85) -- (2.75,5.6);
\draw[dotted] (-4.8,3.5) -- (2.4,5.93);
\draw[dotted] (-4.14,4.18) -- (2.1,6.25);
\draw[dotted] (-3.48,4.84) -- (1.72,6.6);
\draw[dotted] (-2.82,5.5) -- (1.45,6.93);
\draw[dotted] (-2.16,6.16) -- (1.05,7.25);
\draw[dotted] (-1.5,6.82) -- (0.75,7.6);
\draw[dotted] (-0.84,7.5) -- (0.28,7.9);
\node[gray, rotate=18] at (-8.9, -0.35) {\scriptsize $n=12$};
\node[gray, rotate=18] at (-8.35, 0.35) {\scriptsize $n=11$};
\node[gray, rotate=18] at (-7.7, 1) {\scriptsize $n=10$};
\node[gray, rotate=18] at (-7.05, 1.65) {\scriptsize $n=9$};
\node[gray, rotate=18] at (-6.4, 2.3) {\scriptsize $n=8$};
\node[gray, rotate=18] at (-5.75, 2.95) {\scriptsize $n=7$};
\node[gray, rotate=18] at (-5.1, 3.6) {\scriptsize $n=6$};
\node[gray, rotate=18] at (-4.4, 4.3) {\scriptsize $n=5$};
\node[gray, rotate=18] at (-3.75, 4.95) {\scriptsize $n=4$};
\node[gray, rotate=18] at (-3.1, 5.6) {\scriptsize $n=3$};
\node[gray, rotate=18] at (-2.45, 6.25) {\scriptsize $n=2$};
\node[gray, rotate=18] at (-1.75, 6.93) {\scriptsize $n=1$};
\end{tikzpicture}
\caption{The first twelve diagonals of Pascal's triangle}
\label{fig:pascalTri}
\end{figure}


\subsection{A closed form \texorpdfstring{$f_n(\lambda)$}{characteristic polynomial} for the recurrence relation}\label{subsec:closed_form_of_char_poly}

To arrive at the closed form for the solution to the recurrence relation given in Equation~\eqref{eqn:recurrence_relation}, we use our observation of the intimate connection with Pascal's triangle and the coefficients of the individual $f_n(\lambda)$ as $n$ varies. For instance, the $1^{\mathrm{st}}$, $2^{\mathrm{nd}}$, $3^{\mathrm{rd}}$, etc.~coefficients of each $f_n(\lambda)$ coincide directly with entries in the $1^{\mathrm{st}}$, $2^{\mathrm{nd}}$, $3^{\mathrm{rd}}$, etc.~diagonals, respectively, of Pascal's triangle in the particular manner depicted in Figure~\ref{fig:pascalTri}. We highlight this connection to make this relationship apparent.

By careful analysis of the connection between the $f_n(\lambda)$ polynomials and Pascal's triangle, we prove our main result that a closed form for the recurrence relation in Subsection~\ref{subsec:recurrence_relation} is given by
\begin{equation}\label{eqn:general_case}
f_m(\lambda) = \sum_{i=0}^{\left \lfloor{\frac{m}{2}}\right\rfloor } (-1)^{m+i} \binom{m-i}{i} \lambda^{m-2i}.
\end{equation}
In actual practice, it is more helpful to consider the equation above for $m$ even and $m$ odd. The closed form when the index $m$ is even is
\begin{equation}\label{eqn:even_case}
f_{2k}(\lambda) = \sum_{i=0}^k (-1)^i \binom{2k-i}{i} \lambda^{2k-2i},
\end{equation}
and the closed form when the index $m$ is odd is
\begin{equation}\label{eqn:odd_case}
f_{2k+1}(\lambda) = \sum_{i=0}^k (-1)^{i+1} \binom{(2k+1)-i}{i} \lambda^{(2k+1)-2i}.
\end{equation}

Before we prove that the recurrence relation $f_n(\lambda) = -\lambda f_{n-1}(\lambda) - f_{n-2}(\lambda)$ is satisfied by these closed forms above, we give a motivating example for when $n$ is even. The case for when $n$ is odd is similar.

\begin{example}(An even case example)
Set $n = 10$. According to the recurrence relation in Equation~\eqref{eqn:recurrence_relation}, we have
\[
f_{10}(\lambda) = -\lambda f_{9}(\lambda) - f_{8}(\lambda).\label{even_example_case}\tag{$\star$}
\]
By Equation~(\ref{eqn:even_case}), the left hand side of \eqref{even_example_case} is
{\normalsize $$f_{10}(\lambda) = \sum_{i=0}^{5} (-1)^i \binom{10-i}{i} \lambda^{10-2i} = \binom{10}{0}\lambda^{10} - \binom{9}{1}\lambda^{8} + \binom{8}{2}\lambda^{6} - \binom{7}{3}\lambda^{4} + \binom{6}{4}\lambda^{2} - \binom{5}{5}.$$}%
By Equations~(\ref{eqn:even_case}) and~(\ref{eqn:odd_case}), the right hand side of \eqref{even_example_case} is
{\scriptsize
\begin{align*}
\mathlarger{-\lambda f_{9}(\lambda) - f_{8}(\lambda)} &= \mathlarger{-\lambda \mathlarger{\sum}_{i=0}^4 (-1)^{i+1} \binom{9-i}{i} \lambda^{9-2i} - \mathlarger{\sum}_{i=0}^4 (-1)^i \binom{8-i}{i} \lambda^{8-2i}} \\
&= -\lambda\left[ -\binom{9}{0}\lambda^{9} + \binom{8}{1}\lambda^{7} - \binom{7}{2}\lambda^{5} + \binom{6}{3}\lambda^{3} - \binom{5}{4}\lambda \right] - \left[ \binom{8}{0}\lambda^{8} - \binom{7}{1}\lambda^{6} + \binom{6}{2}\lambda^{4} - \binom{5}{3}\lambda^{2} + \binom{4}{4} \right] \\
&= \binom{9}{0}\lambda^{10} - \binom{8}{1}\lambda^{8} + \binom{7}{2}\lambda^{6} - \binom{6}{3}\lambda^{4} + \binom{5}{4}\lambda^{2} - \binom{8}{0}\lambda^{8} + \binom{7}{1}\lambda^{6} - \binom{6}{2}\lambda^{4} + \binom{5}{3}\lambda^{2} - \binom{4}{4} \\
&= \binom{9}{0}\lambda^{10} - \left[ \binom{8}{1} + \binom{8}{0} \right]\lambda^{8} + \left[ \binom{7}{2} + \binom{7}{1} \right]\lambda^{6} - \left[ \binom{6}{3} + \binom{6}{2} \right]\lambda^{4} + \left[ \binom{5}{4} + \binom{5}{3} \right]\lambda^{2} - \binom{4}{4} \\
&= \binom{10}{0}\lambda^{10} - \binom{9}{1}\lambda^{8} + \binom{8}{2}\lambda^{6} - \binom{7}{3}\lambda^{4} + \binom{6}{4}\lambda^{2} - \binom{5}{5}\\
&= \mathlarger{f_{10}(\lambda)}.
\end{align*}}%
The second to last equality above utilizes the fact that $\binom{9}{0} = \binom{10}{0}$ and $\binom{4}{4} = \binom{5}{5}$ trivially, as well as the fundamental combinatorial identity known as Pascal's rule:
\begin{equation}\label{eqn:pascal_prop}
\binom{n}{r} + \binom{n}{r+1} = \binom{n+1}{r+1}.
\end{equation}
\end{example}

We now state and prove the main result of this section.

\begin{theorem}\label{thm:closed_form_satifies_the_recurrence_relation}
Consider the recurrence relation $f_n(\lambda) = -\lambda f_{n-1}(\lambda) - f_{n-2}(\lambda)$ with initial conditions $f_1(\lambda) = -\lambda$ and $f_2(\lambda) = \lambda^2 - 1$. A closed form for this recurrence relation is given by
$$f_m(\lambda) = \sum_{i=0}^{\left \lfloor{\frac{m}{2}}\right\rfloor } (-1)^{m+i} \binom{m-i}{i} \lambda^{m-2i}.$$
\end{theorem}

\begin{proof}
Consider the recurrence relation defined by
$$f_n(\lambda) = -\lambda f_{n-1}(\lambda) - f_{n-2}(\lambda).$$
It is readily verified that Equation~(\ref{eqn:general_case}) yields the initial conditions when $n=1,2$, so it suffices to show that Equation~(\ref{eqn:general_case}) satisfies the recurrence relation. We divide this into two separate cases dependent on the parity of the integer $n \geq 3$.

\noindent \textbf{CASE 1:} ($n$ is even)

\noindent Since $n$ is even with $n \geq 3$, then $n = 2k$ for some $k \geq 2$. Hence we want to show that the following recurrence relation is satisfied by the appropriate choice of closed forms, Equations~(\ref{eqn:even_case}) or~(\ref{eqn:odd_case}), as needed dependent on the parity of each of the three indices below:
\[
f_{2k}(\lambda) = -\lambda f_{2k-1}(\lambda) - f_{2k-2}(\lambda).\label{eqn:proof_even_rr}\tag{$\dagger$}
\]
We start with the right hand side of \eqref{eqn:proof_even_rr}. For the ease of the reader, the expressions in blue indicate the change from one line to the next.
\begingroup 
\allowdisplaybreaks 
{\small
\begin{align}
&\mathlarger{-\lambda f_{2k-1}(\lambda) - f_{2k-2}(\lambda)} \nonumber\\
& \hspace{.5in} = -\lambda \mathlarger{\sum}_{i=0}^{k-1} (-1)^{i+1} \binom{(2k-1)-i}{i} \lambda^{(2k-1)-2i} - \mathlarger{\sum}_{i=0}^{k-1} (-1)^i \binom{(2k-2)-i}{i} \lambda^{(2k-2)-2i} \label{equality_1}\\
& \hspace{.5in} = \mathlarger{\sum}_{i=0}^{k-1} (-1)^{\textcolor{blue}{i}} \binom{(2k-1)-i}{i} \lambda^{\textcolor{blue}{2k-2i}} \textcolor{blue}{\, +} \mathlarger{\sum}_{i=0}^{k-1} (-1)^{\textcolor{blue}{i+1}} \binom{(2k-2)-i}{i} \lambda^{(2k-2)-2i} \nonumber\\
\begin{split}
& \hspace{.5in} = \textcolor{blue}{\binom{2k-1}{0} \lambda^{2k}} + \mathlarger{\sum}_{\textcolor{blue}{i=1}}^{k-1} (-1)^{i} \binom{(2k-1)-i}{i} \lambda^{2k-2i}\\ 
&\hspace{1.8in} + \mathlarger{\sum}_{i=0}^{\textcolor{blue}{k-2}} (-1)^{i+1} \binom{(2k-2)-i}{i} \lambda^{(2k-2)-2i} \textcolor{blue}{+ \, (-1)^k \binom{k-1}{k-1} \lambda^0} \nonumber\\
\end{split}\\
\begin{split}
& \hspace{.5in} = \binom{2k-1}{0} \lambda^{2k} + \mathlarger{\sum}_{\textcolor{blue}{j=0}}^{\textcolor{blue}{k-2}} (-1)^{\textcolor{blue}{j+1}} \binom{(2k-1)-\textcolor{blue}{(j+1)}}{\textcolor{blue}{j+1}} \lambda^{2k-2\textcolor{blue}{(j+1)}}\\ 
&\hspace{1.8in} + \mathlarger{\sum}_{\textcolor{blue}{j=0}}^{k-2} (-1)^{\textcolor{blue}{j}+1} \binom{(2k-2)-\textcolor{blue}{j}}{\textcolor{blue}{j}} \lambda^{(2k-2)-2\textcolor{blue}{j}} + (-1)^k \binom{k-1}{k-1} \lambda^0 \\
\end{split}\label{equality_4}\\
\begin{split}
& \hspace{.5in} = \binom{2k-1}{0} \lambda^{2k} + \textcolor{blue}{\mathlarger{\sum}_{j=0}^{k-2} (-1)^{j+1} \left[ \binom{(2k-1)-(j+1)}{j+1} + \binom{(2k-2)-j}{j} \right] \lambda^{2k-2-2j}}\\
& \hspace{1.8in} + (-1)^k \binom{k-1}{k-1} \lambda^0 \nonumber\\
\end{split}\\
& \hspace{.5in} = \binom{2k-1}{0} \lambda^{2k} + \mathlarger{\sum}_{j=0}^{k-2} (-1)^{j+1} \textcolor{blue}{\binom{(2k-2-j)+1}{j+1}} \lambda^{2k-2-2j} + (-1)^k \binom{k-1}{k-1} \lambda^0 \label{equality_6}\\
& \hspace{.5in} = \binom{\textcolor{blue}{2k}}{0} \lambda^{2k} + \mathlarger{\sum}_{\textcolor{blue}{i=1}}^{\textcolor{blue}{k-1}} (-1)^{\textcolor{blue}{i}} \binom{\textcolor{blue}{2k-i}}{\textcolor{blue}{i}} \lambda^{\textcolor{blue}{2k-2i}} + (-1)^k \binom{\textcolor{blue}{2k-k}}{\textcolor{blue}{k}} \lambda^{\textcolor{blue}{2k-2k}} \label{equality_7}\\
& \hspace{.5in} = \mathlarger{\sum}_{\textcolor{blue}{i=0}}^{\textcolor{blue}{k}} (-1)^{i} \binom{2k-i}{i} \lambda^{2k-2i} \nonumber\\
& \hspace{.5in} = \mathlarger{f_{2k}(\lambda)} \label{equality_9}
\end{align}}%
\endgroup

Line~\eqref{equality_1} holds by the closed forms given by Equations~(\ref{eqn:even_case}) and~(\ref{eqn:odd_case}).
Line~\eqref{equality_4} holds by taking $i=j+1$ for the first summation and $i=j$ for the second summation.
Line~\eqref{equality_6} holds by Equation~(\ref{eqn:pascal_prop}).
Line~\eqref{equality_7} holds by taking $i=j+1$ and realizing that $\binom{2k-1}{0} = \binom{2k}{0}$ and $\binom{k-1}{k-1} = \binom{k}{k}$.
Line~\eqref{equality_9} holds by the closed form given by Equation~(\ref{eqn:even_case}). Therefore the closed forms satisfy the recurrence relation when $n$ is even.
\vspace{.25in}

\noindent \textbf{CASE 2:} ($n$ is odd)

\noindent Since $n$ is odd with $n \geq 3$, then $n = 2k + 1$ for some $k \geq 1$. Hence we want to show that the following recurrence relation is satisfied by the appropriate choice of closed forms, Equations~(\ref{eqn:even_case}) or~(\ref{eqn:odd_case}), as needed dependent on the parity of each of the three indices below:
\[
f_{2k+1}(\lambda) = -\lambda f_{2k}(\lambda) - f_{2k-1}(\lambda).\label{eqn:proof_odd_rr}\tag{$\dagger\dagger$}
\]
We start with the right hand side of \eqref{eqn:proof_odd_rr}:
\begingroup
\allowdisplaybreaks
{\small
\begin{align}
&\mathlarger{-\lambda f_{2k}(\lambda) - f_{2k-1}(\lambda)} \nonumber\\
& \hspace{.5in} = -\lambda \mathlarger{\sum}_{i=0}^{k} (-1)^{i} \binom{2k-i}{i} \lambda^{2k-2i} - \mathlarger{\sum}_{i=0}^{k-1} (-1)^{i+1} \binom{(2k-1)-i}{i} \lambda^{(2k-1)-2i} \label{equality_10}\\
& \hspace{.5in} = \mathlarger{\sum}_{i=0}^{k} (-1)^{\textcolor{blue}{i+1}} \binom{2k-i}{i} \lambda^{\textcolor{blue}{2k-2i+1}} \textcolor{blue}{+} \mathlarger{\sum}_{i=0}^{k-1} (-1)^{\textcolor{blue}{i}} \binom{(2k-1)-i}{i} \lambda^{(2k-1)-2i} \nonumber\\
\begin{split}
& \hspace{.5in} = \textcolor{blue}{-\binom{2k}{0} \lambda^{2k+1}} + \mathlarger{\sum}_{\textcolor{blue}{i=1}}^{k} (-1)^{i+1} \binom{2k-i}{i} \lambda^{2k-2i+1}\\ 
&\hspace{1.8in} + \mathlarger{\sum}_{\textcolor{blue}{j=1}}^{\textcolor{blue}{k}} (-1)^{\textcolor{blue}{j-1}} \binom{(2k-1)-\textcolor{blue}{(j-1)}}{\textcolor{blue}{j-1}} \lambda^{(2k-1)-2\textcolor{blue}{(j-1)}} \\
\end{split} \label{equality_12}\\
& \hspace{.5in} = -\binom{2k}{0} \lambda^{2k+1} + \mathlarger{\sum}_{\textcolor{blue}{j=1}}^{k} (-1)^{\textcolor{blue}{j}+1} \binom{2k-\textcolor{blue}{j}}{\textcolor{blue}{j}} \lambda^{2k-2\textcolor{blue}{j}+1} + \mathlarger{\sum}_{j=1}^{k} (-1)^{\textcolor{blue}{j+1}} \binom{\textcolor{blue}{2k-j}}{j-1} \lambda^{\textcolor{blue}{2k-2j+1}} \label{equality_13}\\
& \hspace{.5in} = -\binom{2k}{0} \lambda^{2k+1} + \textcolor{blue}{\mathlarger{\sum}_{j=1}^{k} (-1)^{j+1} \left[ \binom{2k-j}{j} + \binom{2k-j}{j-1} \right] \lambda^{2k-2j+1}} \nonumber\\
& \hspace{.5in} = -\binom{\textcolor{blue}{2k+1}}{0} \lambda^{2k+1} + \mathlarger{\sum}_{j=1}^{k} (-1)^{j+1} \textcolor{blue}{\binom{(2k+1)-j}{j}} \lambda^{\textcolor{blue}{(2k+1)-2j}} \label{equality_15}\\
& \hspace{.5in} = \mathlarger{\sum}_{\textcolor{blue}{j=0}}^{k} (-1)^{j+1} \binom{(2k+1)-j}{j} \lambda^{(2k+1)-2j} \nonumber\\
& \hspace{.5in} = \mathlarger{f_{2k+1}(\lambda)} \label{equality_17}
\end{align}}%
\endgroup

Line~\eqref{equality_10} holds by the closed forms given by Equations~(\ref{eqn:even_case}) and~(\ref{eqn:odd_case}).
Line~\eqref{equality_12} holds by taking $i=j-1$ for the second summation.
Line~\eqref{equality_13} holds by taking $i=j$ for the first summation.
Line~\eqref{equality_15} holds by Equation~(\ref{eqn:pascal_prop}) and realizing that $\binom{2k}{0} = \binom{2k+1}{0}$.
Line~\eqref{equality_17} holds by the closed form given by Equation~(\ref{eqn:odd_case}). Therefore the closed forms satisfy the recurrence relation when $n$ is odd. Hence Equation~(\ref{eqn:general_case}) satisfies the recurrence relation, and the result follows.
\end{proof}


\subsection{Parity properties of the \texorpdfstring{$f_n(\lambda)$}{the characteristic polynomial}}\label{subsec:some_properties}

Observing the graphs of various characteristic functions $f_n(\lambda)$ in Figure~\ref{fig:plots_of_various_equations}, it appears that the functions of even index are symmetric about the $y$-axis (the hallmark trait of an \textit{even function}), while the functions of odd index are symmetric about the origin---that is, the graph remains unchanged after 180 degree rotation about the origin (the hallmark trait of an \textit{odd function}). This observation prompts the following theorem.
\begin{theorem}\label{thm:parity_of_char_polys}
The characteristic polynomial $f_n(\lambda)$ is an even (respectively, odd) function when $n$ is even (respectively, odd).
\end{theorem}

\begin{proof}
Clearly $f_1(-\lambda) = -f_1(\lambda)$ and $f_2(-\lambda) = f_2(\lambda)$, so it suffices to justify the following:
\begin{itemize}
\item Claim 1: $f_{2k}(-\lambda) = f_{2k}(\lambda)$ for all $k \geq 2$, and
\item Claim 2: $f_{2k+1}(-\lambda) = -f_{2k+1}(\lambda)$ for all $k \geq 1$.
\end{itemize}
By Equation~\eqref{eqn:even_case}, we have $f_{2k}(-\lambda) = f_{2k}(\lambda)$ if and only if $(-\lambda)^{2k-2i} = \lambda^{2k-2i}$. But since $2k - 2i$ is even, it is clear that $(-\lambda)^{2k-2i} = \lambda^{2k-2i}$ holds, so Claim 1 follows. On the other hand, by Equation~\eqref{eqn:odd_case}, we have $f_{2k+1}(-\lambda) = -f_{2k+1}(\lambda)$ if and only if $(-\lambda)^{(2k+1)-2i} = -\lambda^{(2k+1)-2i}$. But since $(2k+1)-2i$ is odd, it is clear that  $(-\lambda)^{(2k+1)-2i} = -\lambda^{(2k+1)-2i}$ holds, so Claim 2 follows.
\end{proof}


\subsection{A connection to Chebyshev polynomials of the second kind}\label{subsec:Chebyshev}

In personal communication with Paul Terwilliger, it was brought to our attention that our characteristic polynomials $f_n(\lambda)$ correspond to a certain normalization of Chebyshev polynomials of the second kind (up to a change of variable for $-\lambda$ in place of $\lambda$). We first recall the definition of these well-known polynomials, and then we give the normalization.

\begin{definition}
The sequence $(U_n(x))_{n \geq 0}$ of \textit{Chebyshev polynomials of the second kind} are given by the recurrence relation $U_n(x) = 2x \cdot U_{n-1}(x) - U_{n-2}(x)$
for $n \geq 2$ with initial conditions $U_0(x) = 1$ and $U_1(x) = 2x$.
\end{definition}

The following normalization of the Chebyshev polynomials of the second kind is well known  and appeared as early as 1646 in Vi\`ete's \textit{Opera Mathematica} (Chapter IX, Theorem VII), but we take the following definition and notation from the National Bureau of Standards in 1952~\cite{NBS1952}.

\begin{definition}
The sequence $(S_n(x))_{n \geq 0}$ of \textit{normalized Chebyshev polynomials of the second kind} are given by the recurrence relation $S_n(x) = x \cdot S_{n-1}(x) - S_{n-2}(x)$ for all $n \geq 2$ with initial conditions $S_0(x) = 1$ and $S_1(x) = x$.
\end{definition}

In Table~\ref{table:Chebshev_polys}, we present the Chebyshev polynomials of the second kind $U_n(x)$, normalized Chebyshev polynomials $S_n(x)$, and our characteristic polynomials $f_n(x)$ for the first seven $n$-index values.
\begin{table}[h!]
\[
\arraycolsep=5pt\def\arraystretch{1.2}
\begin{array}{|c||c|c|c|}
\hline
 n & U_n(x) & S_n(x) & f_n(x)\\
\hline 
\hline
0 & 1 & 1 & 1\\ \hline
1 & 2x & x & -x\\ \hline
2 & 4x^2 - 1 & x^2 - 1 & x^2 - 1\\ \hline
3 & 8 x^3 - 4x & x^3 - 2x & -x^3 + 2x\\ \hline
4 & 16 x^4 - 12 x^2 + 1 & x^4  - 3 x^2 + 1 & x^4  - 3 x^2 + 1\\ \hline
5 & 32 x^5 - 32 x^3 + 6x &  x^5 - 4 x^3 + 3x & -x^5 + 4 x^3 - 3x\\ \hline
6 & 64 x^6 - 80 x^4 + 24 x^2 - 1 &  x^6 - 5 x^4 + 6 x^2 - 1 & x^6 - 5 x^4 + 6 x^2 - 1\\
\hline
\end{array}
\]
\caption{Comparison of the first seven polynomials $U_n(x)$, $S_n(x)$, and $f_n(x)$}
\label{table:Chebshev_polys}
\end{table}

\noindent It is clear that we have the following connection between our family of characteristic polynomials and the normalized Chebyshev and Chebyshev polynomials:
$$ f_n(\lambda) =  S_n(-\lambda) = U_n\left( -\frac{\lambda}{2} \right).$$

\begin{remark}
The polynomials $S_n(x)$ are also known as \textit{Vieta-Fibonacci polynomials} denoted by $V_n(x)$ and introduced as such by Horadam~\cite{Horadam2002}.
\end{remark}


\section{The spectrum of the matrix \texorpdfstring{$A_n$}{the matrices}}\label{sec:spectrum}

As mentioned in the introduction, the $A_n$ matrices are the adjacency matrices of a very important class of graphs, namely the path graphs. For example, consider the path graph $P_4$ and its corresponding adjacency matrix, which is the matrix $A_4$.

\begin{center}
\begin{tikzpicture}
	\fill (-3,0) circle (2pt);
	\node at (-3,0.3) {\footnotesize$v_1$};
	\fill (-1.5,0) circle (2pt);
	\node at (-1.5,0.3) {\footnotesize$v_2$};
	\fill (0,0) circle (2pt);
	\node at (0,0.3) {\footnotesize$v_3$};
	\fill (1.5,0) circle (2pt);
	\node at (1.5,0.3) {\footnotesize$v_4$};
	\draw[black,thick] (-3,0) to (-1.5,0);
	\draw[black,thick] (-1.5,0) to (0,0);
	\draw[black,thick] (0,0) to (1.5,0);
	\node at (2.5,-0.05) {$\longleftrightarrow$};
	\node at (5,0.25) {$\kbordermatrix{
    & v_1 & v_2 & v_3 & v_4 \\
    v_1 & 0 & 1 & 0 & 0 \\
    v_2 & 1 & 0 & 1 & 0 \\
    v_3 & 0 & 1 & 0 & 1 \\
    v_4 & 0 & 0 & 1 & 0  
  }$};
\end{tikzpicture}	
\end{center}
These matrices are fundamental objects in spectral graph theory. In this section, we examine the spectrum of these matrices. Here is a table of the exact roots of the first five characteristic equations, and below that is a table of these roots approximated up to 5 decimal places.
{\small
$$
\begin{array}{|c||l|}
\hline
 n\text{-value} & \text{exact roots of } f_n(\lambda) = 0\\
\hline \hline
 1 & \lambda_1 =0 \\
 2 & \lambda_1 =1\text{ or } \lambda_1 =-1 \\
 3 & \lambda_1 =0\text{ or } \lambda_2 =\sqrt{2}\text{ or } \lambda_3 =-\sqrt{2} \\
 4 & \lambda_1 =\frac{1}{2} \left(-\sqrt{5}-1\right)\text{ or } \lambda_2 =\frac{1}{2} \left(\sqrt{5}-1\right)\text{ or } \lambda_3 =\frac{1}{2} \left(1-\sqrt{5}\right)\text{ or } \lambda_4 =\frac{1}{2} \left(\sqrt{5}+1\right) \\
 5 & \lambda_1 =0\text{ or } \lambda_2 =\sqrt{3}\text{ or } \lambda_3 =-\sqrt{3}\text{ or } \lambda_4 =-1\text{ or } \lambda_5 =1 \\
\hline
\end{array}
$$
}

{\small
$$
\begin{array}{|c||l|}
\hline
 n\text{-value} & \text{roots (up to 5 decimal places) of } f_n(\lambda) = 0 \text{ in increasing order}\\
\hline \hline
 1 & \lambda_1 =0 \\
 2 & \lambda_1 =-1\text{ or } \lambda_2 =1 \\
 3 & \lambda_1 =-1.41421\text{ or } \lambda_2 =0\text{ or } \lambda_3 =1.41421 \\
 4 & \lambda_1 =-1.61803\text{ or } \lambda_2 =-0.618034\text{ or } \lambda_3 =0.618034\text{ or } \lambda_4 =1.61803 \\
 5 & \lambda_1 =-1.73205\text{ or } \lambda_2 =-1\text{ or } \lambda_3 =0\text{ or } \lambda_3 =1\text{ or } \lambda_5 =1.73205 \\
\hline
\end{array}
$$
}

In Figure~\ref{fig:plots_of_various_equations}, we plot the characteristic polynomials $f_n(\lambda)$ for $n = 2, 3, 4, 5$. Notice that since the odd-index equations have no constant term, the graphs of those equations necessarily pass through the origin.

\begin{figure}[H]
\begin{center}
\includegraphics[height=2.25in]{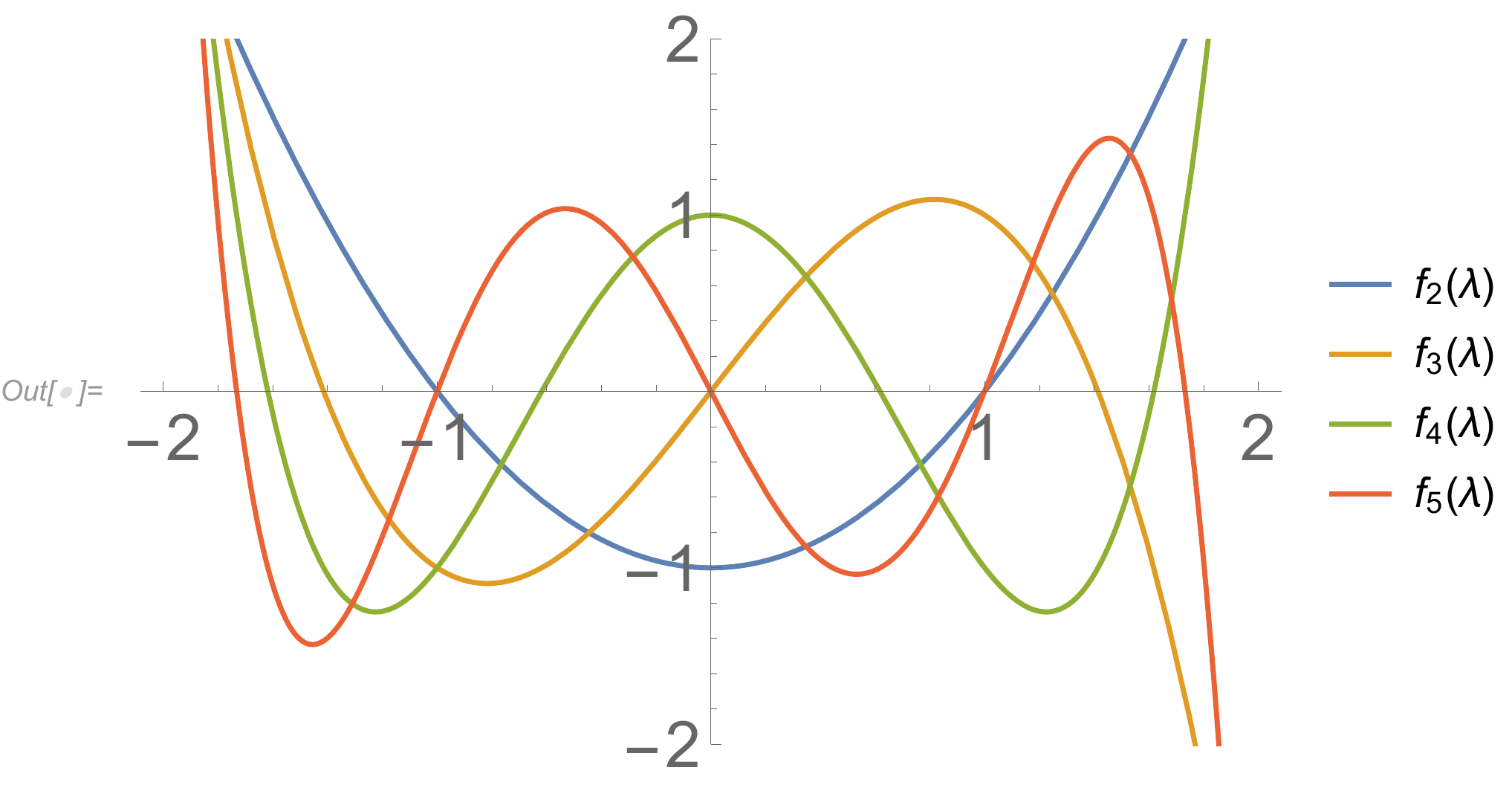}
\caption{Plots of $f_n(\lambda)$ for $n = 2, 3, 4, 5$}
\label{fig:plots_of_various_equations}
\end{center}
\end{figure}
\noindent Examining the roots of $f_n(\lambda)=0$ in just these few small cases, we can visually confirm the following:
\begin{itemize}
\item As Theorem~\ref{thm:parity_of_char_polys} guarantees, the polynomials $f_3$ and $f_5$ are odd functions while the polynomials $f_2$ and $f_4$ are even functions.
\item If $\lambda_i$ is a root of $f_n(\lambda) = 0$, then $-\lambda_i$ is also a root.
\item For $n$ even, all $n$ roots are nonzero and come in distinct pairs $\lambda_i$ and $-\lambda_i$ for $i = 1, 2, \ldots, \frac{n}{2}$.
\item For $n$ odd, exactly one root is zero and the other $n-1$ roots come in distinct pairs $\lambda_i$ and $-\lambda_i$ for $i = 1, 2, \ldots, \frac{n-1}{2}$.
\item As $n$ increases, the largest eigenvalue of the matrix $A_n$ also increases and appears to be bounded above by 2. More precisely, it seems that $\rho(A_n) < 2$ for all $n$.
\end{itemize}
In this section, we explore in generality all the observations above. Regarding the spectrum of the family $\{A_n\}_{n=1}^\infty$, we investigate the spectral radius, lower and upper bounds of $\spec(A_n)$, and a spectrum set-containment sufficiency condition for $\spec(A_m) \subset \spec(A_n)$ dependent on the values $m,n \in \mathbb{N}$. We prove many of our assertions using a trigonometric closed form for the roots of $f_n(\lambda)$.


\subsection{A trigonometric closed form for the roots of \texorpdfstring{$f_n(\lambda) = 0$}{the characteristic equations}}\label{subsec:closed_form_for_roots}
Though many sources going back as far as 1965 are credited in the literature as giving closed forms for the eigenvalues of tridiagonal symmetric matrices, few of these cited papers or books actually provide a proof~\cite{Abromovitz1965,Barnett1990,Kahan1966}. However in his 1978 book, Smith provides a proof for the closed forms for eigenvalues of tridiagonal (but not necessarily) symmetric matrices that have the values $a,b,c \in \mathbb{C}$ in the diagonal, superdiagonal, and subdiagonal, respectively~\cite{Smith1978}. For these matrix values, the $n$ eigenvalues are given by the formula
$$\lambda_s = a + 2\sqrt{bc} \cos\left(\frac{s \pi}{n+1}\right) \text{ for } s=1,\ldots,n.$$
Below we modify the proof to give a closed form expression for the eigenvalues of the $A_n$ matrices.

\begin{proposition}\label{prop:closed_form_for_eigenvalues}
The $n$ distinct eigenvalues of the matrix $A_n$ are
$$\lambda_s = 2\cos\left( \frac{s \pi}{n+1} \right) \text{ for } s=1,\ldots,n.$$
\end{proposition}

\begin{proof}
Let $\lambda$ represent an eigenvalue of the matrix $A_n$ with corresponding eigenvector $\vec{v}$. Consider the equation $A_n\vec{v} = \lambda\vec{v}$. This can be rewritten as $\left( A_n - \lambda I_n \right) \vec{v} = \vec{0}$. Then $$\left( A_n - \lambda I_n \right) \vec{v} = 
\begin{bmatrix} 
-\lambda & 1 &  & \\
1 & \ddots & \ddots & \\
 & \ddots & \ddots & 1\\
 &  & 1 & -\lambda
\end{bmatrix}
\begin{bmatrix}
\vec{v}_1 \\
\vec{v}_2 \\
\vdots \\
\vec{v}_n
\end{bmatrix}
= \begin{bmatrix}
0 \\
0 \\
\vdots \\
0
\end{bmatrix}.
$$
The resulting system of equations is as follows:
\begin{align*}
-\lambda\vec{v}_1 + \vec{v}_2 &= 0\\
\vec{v}_1 - \lambda\vec{v}_2 + \vec{v}_3 &= 0\\
\vdots \quad \quad & \quad \, \, \vdots\\
\vec{v}_{j-1} - \lambda\vec{v}_{j} + \vec{v}_{j+1} &= 0\\
\vdots \quad \quad & \quad \, \, \vdots\\
\vec{v}_{n-2} - \lambda\vec{v}_{n-1} + \vec{v}_{n} &= 0\\
\vec{v}_{n-1} - \lambda\vec{v}_n &= 0
\end{align*}
In general, a single equation can be defined by 
\begin{equation}\label{eigenvector_eqn}
\vec{v}_{j-1} - \lambda\vec{v}_{j} + \vec{v}_{j+1} = 0 \quad \text{ for } j=1,\hdots,n
\end{equation}
 where $\vec{v}_0 = 0$ and $\vec{v}_{n+1} = 0$. Let a solution to this equation be $\vec{v}_j=Am^j$ for arbitrary nonzero constants $A$ and $m$. Substituting this solution into Equation~\eqref{eigenvector_eqn} shows that $m$ is a root of 
$$m^2 - \lambda m + 1 = 0.$$
Since $m^2 - \lambda m + 1 = 0$ is a quadratic equation, there are two, in our case real, roots of this quadratic. Let us denote these two roots by $m_1$ and $m_2$. Hence the general solution of Equation~\eqref{eigenvector_eqn} is 
\begin{equation}\label{general_solution}
\vec{v}_j = Bm_1^j + Cm_2^j
\end{equation}
where $B$ and $C$ are arbitrary nonzero constants. Since $\vec{v}_0 = 0$ and $\vec{v}_{n+1} = 0$, we get $B+C=0$ and $Bm_1^{n+1} + Cm_2^{n+1}=0$, respectively, by Equation~\eqref{general_solution}. Substituting the former equation into the latter results in
$$\left( \frac{m_1}{m_2} \right)^{n+1} = 1 = e^{2\pi i s} \quad \text{ for } s = 1,\hdots,n.$$ Hence it follows that
\begin{equation}\label{rootsFrac_eqn}
\frac{m_1}{m_2} = e^{2\pi i s/(n+1)}.
\end{equation}
Since $m^2 - \lambda m + 1 = 0$ is a quadratic equation, the product of the roots is given by $m_1m_2 = 1$. Using Equation~\eqref{rootsFrac_eqn}, we have the sequence of implications
\begin{align*}
\frac{m_1}{m_2} = e^{2\pi i s/(n+1)} &\Longrightarrow m_1 = m_2 \cdot e^{2\pi i s/(n+1)}\\
&\Longrightarrow m_1^2 = m_1m_2 \cdot e^{2\pi i s/(n+1)}\\
&\Longrightarrow m_1^2 = e^{2\pi i s/(n+1)}	&\mbox{since $m_1 m_2 = 1$}\\
&\Longrightarrow m_1 = e^{\pi i s/(n+1)}.
\end{align*}
Through a similar argument it can be found that $$m_2 = e^{-\pi i s/(n+1)}.$$
Again since $m^2 - \lambda m + 1 = 0$ is a quadratic equation, the sum of the roots is given by $m_1 + m_2 = \lambda$. It follows that
\begin{align*}
\lambda &= m_1 + m_2\\
&= e^{\pi i s/(n+1)} + e^{-\pi i s/(n+1)} \\
&= e^{i\left(\frac{\pi s}{n+1}\right)} + e^{i\left(\frac{-\pi s}{n+1}\right)} \\
&= \left( \cos \left(\frac{\pi s}{n+1}\right) + i \sin \left(\frac{\pi s}{n+1}\right)\right) + \left( \cos \left(\frac{-\pi s}{n+1}\right) + i \sin \left(\frac{-\pi s}{n+1}\right)\right) \\
&= 2\cos{\left( \frac{\pi s}{n+1} \right)},
\end{align*}
where the last equality holds since cosine is an even function and sine is an odd function. Thus we conclude that $\lambda_s = 2\cos{\left( \frac{\pi s}{n+1} \right)}$ for $s = 1,\hdots,n$.
\end{proof}


\subsection{Spectral radius of the matrix \texorpdfstring{$A_n$}{the matrices} and some boundary values}\label{subsec:spectral_radius}

Much work has been done to study the bounds of the eigenvalues of real symmetric matrices. Methods discovered in the mid-nineteenth century reduce the original matrix to a tridiagonal matrix whose eigenvalues are the same as those of the original matrix. Exploiting this idea, Golub in 1962 determined lower bounds on tridiagonal matrices of the form
$$
\begin{bmatrix} 
  a_0 & b_1 & 0 & \ldots & \ldots & 0 \\
  b_1 & a_2 & b_2 & \ddots & & \vdots \\
  0 & b_2 & \ddots & \ddots & \ddots& \vdots \\
  \vdots & \ddots & \ddots & \ddots & b_{n-2} & 0\\
  \vdots & & \ddots & b_{n-2} & a_{n-1} & b_{n-1} \\
  0 & \ldots & \ldots & 0 & b_{n-1} & a_n 
\end{bmatrix}.
$$

\begin{proposition}[Golub~{\cite[Corollary~1.1]{Golub1962}}]\label{prop:Golub_result}
Let $A$ be an $n \times n$ matrix with real entries $a_{ij} = a_i$ for $i=j$, $a_{ij} = b_m$ for $|i-j|=1$ where $m = \min(i,j)$, and $a_{ij} = 0$ otherwise. Then the interval $[a_k - \sigma_k, \; a_k + \sigma_k]$ where $\sigma_k^2 = b_k^2 + b_{k-1}^2$ contains at least one eigenvalue.
\end{proposition}

Applying the latter proposition to the tridiagonal $A_n$ matrices, we easily yield an interval in which the lower bound (in absolute value) of the eigenvalues of $A_n$ is guaranteed to occur.

\begin{corollary}
For each matrix $A_n$, there exists an eigenvalue $\lambda$ such that $|\lambda| \leq 1$.
\end{corollary}

\begin{proof}
Consider the matrix $A_n$. Then in the language of Proposition~\ref{prop:Golub_result}, we have $a_i = 0$ and $b_i = 1$ for each $i$. Thus the $\sigma_k$ are calculated as follows
$$
\sigma_k =
\begin{cases}
	1 & \mbox{if } k=1, \\ 
	\sqrt{2} & \mbox{if } 2 \leq k \leq n-1, \\
	1 & \mbox{if } k=n.
\end{cases}
$$
By Proposition~\ref{prop:Golub_result}, an eigenvalue is guaranteed to exist in the the interval $[a_k - \sigma_k, \; a_k + \sigma_k]$. In particular for $k=1$ the result follows.
\end{proof}

So we have a firm interval where the lower bound (in absolute value) will contain an eigenvalue of $A_n$. As for an upper bound, it is clear by Proposition~\ref{prop:closed_form_for_eigenvalues} that the following theorem holds.

\begin{theorem}
For each $n \in \mathbb{N}$, the spectral radius $\rho(A_n)$ of the matrix $A_n$ is bounded above by $2$.
\end{theorem}


\subsection{Sufficient condition for \texorpdfstring{$\spec(A_m) \subset \spec(A_n)$}{spectrum containments}}\label{subsec:suff_cond}

Recall in Section~\ref{sec:char_polys}, we gave the characteristic polynomial $f_n(\lambda)$ of the matrix $A_n$ when $n=7$ as follows:
$$ f_7(\lambda) = -\lambda ^7+6 \lambda ^5-10 \lambda ^3+4 \lambda.$$
The characteristic equation $f_7(\lambda) = 0$ has the seven roots
$$
\begin{array}{l}
\lambda =0\\
\lambda =\pm\sqrt{2}\\
\lambda =\pm\sqrt{2-\sqrt{2}}\\
\lambda =\pm\sqrt{\sqrt{2}+2}.
\end{array}
$$
It is interesting to note that these distinct seven roots appear as eigenvalues in a higher-degree characteristic polynomial. For example, when $n=15$, the characteristic polynomial is
$$f_{15}(\lambda) = -\lambda ^{15}+14 \lambda ^{13}-78 \lambda ^{11}+220 \lambda ^9-330 \lambda ^7+252 \lambda ^5-84 \lambda ^3+8 \lambda.$$
The characteristic equation $f_{15}(\lambda) = 0$ has the 15 roots
$$
\begin{array}{ll}
\lambda =0 & \lambda =\pm\sqrt{2-\sqrt{2-\sqrt{2}}}\\
\lambda =\pm\sqrt{2} & \lambda =\pm\sqrt{\sqrt{2-\sqrt{2}}+2}\\
\lambda =\pm\sqrt{2-\sqrt{2}} & \lambda =\pm\sqrt{2-\sqrt{\sqrt{2}+2}}\\
\lambda =\pm\sqrt{\sqrt{2}+2} & \lambda =\pm\sqrt{\sqrt{\sqrt{2}+2}+2}.
\end{array}
$$

This phenomenon is no coincidence. In fact the following theorem provides a way to predict the values $m < n$ for which a complete set of roots of the equation $f_m(\lambda)=0$ will be contained in the set of roots of the equation $f_n(\lambda)=0$.

\begin{theorem}\label{thm:spec}
Fix $m \in \mathbb{N}$. For any $n \in \mathbb{N}$ such that $m<n$ and $n \equiv m \pmod{m+1}$, it follows that $\spec(A_m) \subset \spec(A_n)$.
\end{theorem}

\begin{proof}
Let $m,n \in \mathbb{N}$ such that $m<n$ and $n \equiv m \pmod {m+1}$. Then $n - m = (m+1)k$ for some $k \in \mathbb{N}$. This can be rewritten as 
\begin{equation}\label{eqn:thmspec_assumption}
n = m(k+1) + k.
\end{equation}
Now consider $\spec(A_m)$ and $\spec(A_n)$. By Proposition~\ref{prop:closed_form_for_eigenvalues}, these are defined as follows:
\begin{center}
$\spec(A_m) = \left\lbrace \lambda_r \mid \lambda_r = 2\cos\left( \frac{r \pi}{m+1} \right)\right\rbrace_{r=1}^m$\\
$\spec(A_n) = \left\lbrace \lambda_s \mid \lambda_s = 2\cos\left( \frac{s \pi }{n+1} \right)\right\rbrace_{s=1}^n.$
\end{center}
By the following sequence of equalities, we can rewrite $\lambda_s \in \spec(A_n)$ in terms of $m$ and $k$.
\begingroup
\allowdisplaybreaks
{
\begin{align*}
\lambda_s &= 2\cos\left( \frac{s \pi}{n+1} \right) \text{ for  } s = 1, 2, \ldots, n\\
&= 2\cos\left( \frac{s \pi}{m(k+1)+k+1} \right) \text{ for  } s = 1, 2, \ldots, m(k+1) + k		&\mbox{by Equation~\eqref{eqn:thmspec_assumption}}\\
&= 2\cos\left( \frac{s \pi}{mk+m+k+1} \right)\\
&= 2\cos\left( \frac{s \pi}{(m+1)(k+1)} \right)\\
&= 2\cos\left( \frac{\left( \frac{s}{k+1} \right) \pi}{m+1} \right).
\end{align*}
}
\endgroup
Recall $\lambda_r = 2\cos\left( \frac{r \pi}{m+1} \right)$. Hence $\lambda_r = \lambda_s$ when $r = \frac{s}{k+1}$. 

Since $s \in S = \{ 1, 2, \ldots, m(k+1) + k \}$, the set $T = \{ k+1, 2(k+1), \ldots, m(k+1)\}$ is a subset of $S$. Notice for each $s \in T$, there exists a unique $r \in \{ 1, 2, \ldots, m\}$ under the equality $r = \frac{s}{k+1}$. Thus for all $\lambda_r \in \spec(A_m)$, there exists a $\lambda_s \in \spec(A_n)$ such that $\lambda_r = \lambda_s$. Then $\lambda_r \in \spec(A_n)$ for all $r \in \{ 1, 2, \ldots, m\}$. Therefore $\spec(A_m) \subset \spec(A_n)$.
\end{proof}

\begin{remark}
In the introduction we noted our observation via \texttt{Mathematica} that the golden ratio, its reciprocal, and their additive inverses arise as eigenvalues of $A_n$ for the $n$ values 4, 9, 14, 19, 24, 29, 34, 39, and 44. In Example~\ref{exam:golden_ratio}, we computed the four eigenvalues of $A_4$. As an application of Theorem~\ref{thm:spec}, if we let $m=4$, then it is evident that the $\spec(A_4) \subset \spec(A_{4+5k})$ for all $k \in \mathbb{N}$, thus confirming our observation.
\end{remark}

\section{Open questions}\label{sec:open_questions}

\begin{question}
In the plot in Figure~\ref{fig:plots_of_various_equations} it appears that the maximum and minimum values for the characteristic polynomials lie on a hyperbola. For example, here is a plot of $f_{10}(\lambda)$ and $f_{20}(\lambda)$. 
\begin{center}
\includegraphics[height=1.75in]{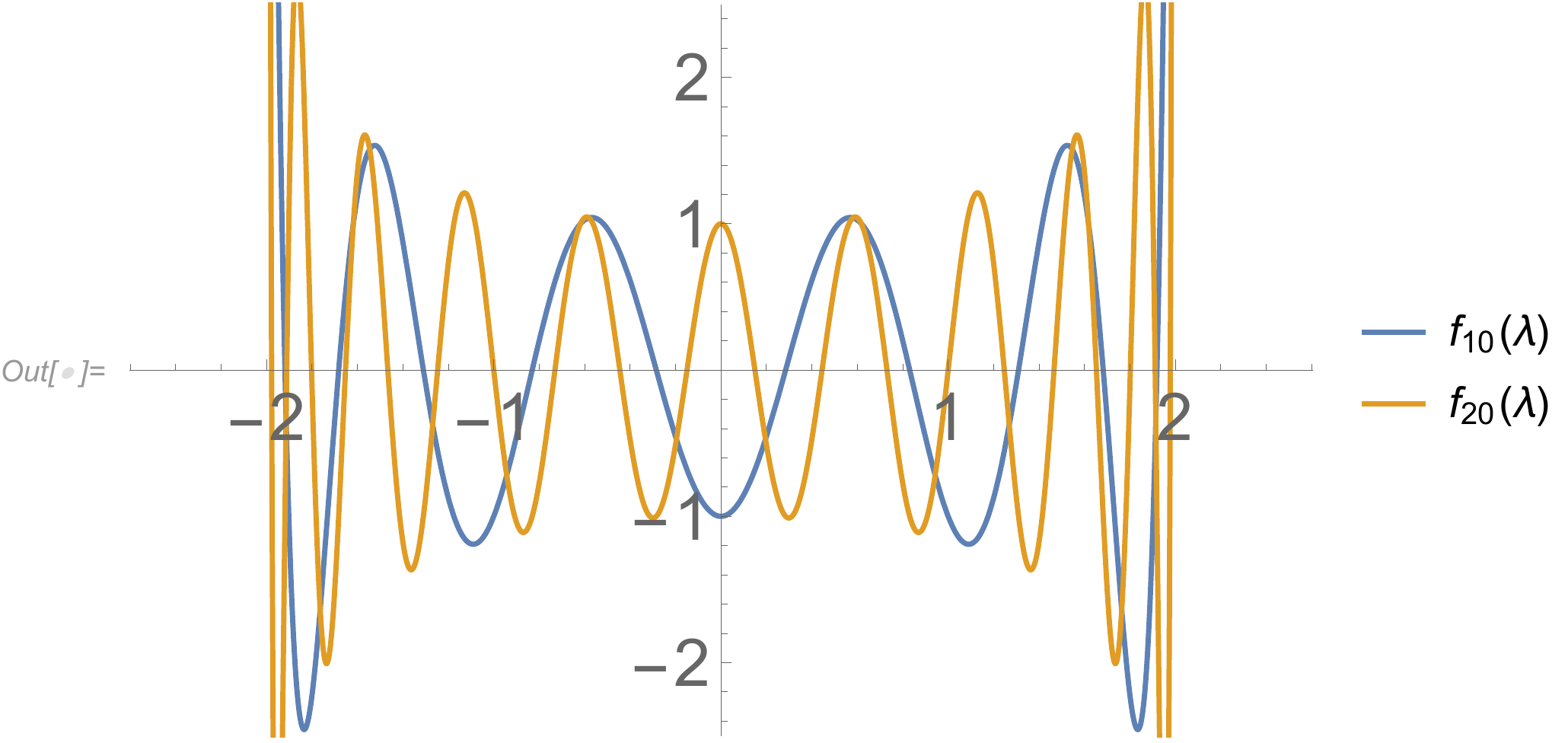}
\end{center}
Do the max and min values of the sequence $\{f_n(\lambda)\}_{n=1}^\infty$  lie on a particular hyperbola? And if so, can we find an exact formula for this hyperbola?
\end{question}

\begin{question}
A natural setting in which to generalize this research is the following family of $n \times n$ matrices
$$
\begin{bmatrix} 
0 & b &  & \\
b & \ddots & \ddots & \\
 & \ddots & \ddots & b\\
 &  & b & 0
\end{bmatrix}
$$
where $b>0$ is some integer. These are the adjacency matrices of the multiple-edged path graphs, that is, the graphs that have $b$ edges between each pair of consecutive vertices. For example, below we give the multiple-edge path graph on 4 vertices with two edges between each vertex, and we give its corresponding adjacency matrix.

\begin{center}
\begin{tikzpicture}
	\fill (-3,0) circle (2pt);
	\node at (-3,0.3) {\footnotesize$v_1$};
	\fill (-1.5,0) circle (2pt);
	\node at (-1.5,0.3) {\footnotesize$v_2$};
	\fill (0,0) circle (2pt);
	\node at (0,0.3) {\footnotesize$v_3$};
	\fill (1.5,0) circle (2pt);
	\node at (1.5,0.3) {\footnotesize$v_4$};
	\draw[bend right] (-3,0) to (-1.5,0);
	\draw[bend right] (-1.5,0) to (0,0);
	\draw[bend right] (0,0) to (1.5,0);
	\draw[bend left] (-3,0) to (-1.5,0);
	\draw[bend left] (-1.5,0) to (0,0);
	\draw[bend left] (0,0) to (1.5,0);
	\node at (2.5,-0.05) {$\longleftrightarrow$};
	\node at (5,0.25) {$\kbordermatrix{
    & v_1 & v_2 & v_3 & v_4 \\
    v_1 & 0 & 2 & 0 & 0 \\
    v_2 & 2 & 0 & 2 & 0 \\
    v_3 & 0 & 2 & 0 & 2 \\
    v_4 & 0 & 0 & 2 & 0  
  }$};
	\end{tikzpicture}	
\end{center}
Do the coefficients of the corresponding characteristic polynomials have any connection to the binomial coefficients as they do in the $b=1$ case? Moreover, is there an analogue of the spectrum containment sufficiency condition as in our Theorem~\ref{thm:spec}?
\end{question}


\subsection{Conjectures on \texorpdfstring{$f_n(\lambda) = F_{n+1}$}{fn(lambda)=Fn+1} where \texorpdfstring{$F_{n+1}$}{Fn+1} is the \texorpdfstring{$(n+1)^{\mathrm{th}}$}{(n+1)th} Fibonacci number}

There is a tantalizing connection between the characteristic polynomials $f_n(\lambda)$ and the Fibonacci sequence $\{F_n\}_{n=0}^\infty = \{0, 1, 1, 2, 3, 5, 8, 13, 21, \ldots\}$. Recall the polynomial $f_{12}(\lambda)$ from Table~\ref{table:char_polys}.
$$f_{12}(\lambda) = \textcolor{black}{1}\lambda ^{12}  - \textcolor{black}{11} \lambda ^{10}  + \textcolor{black}{45} \lambda ^8  - \textcolor{black}{84} \lambda ^6  + \textcolor{black}{70} \lambda ^4  - \textcolor{black}{21} \lambda ^2  + \textcolor{black}{1}.$$
For what $\lambda$ does $f_{12}(\lambda) = F_{13}$? It turns out $\lambda = i$ is a root as follows:
\begin{align*}
    f_{12}(i) &= \textcolor{black}{1} i^{12}  - \textcolor{black}{11} i^{10}  + \textcolor{black}{45} i^8  - \textcolor{black}{84} i^6  + \textcolor{black}{70} i^4  - \textcolor{black}{21} i^2  + \textcolor{black}{1}\\
    &= \textcolor{black}{1} + \textcolor{black}{11} + \textcolor{black}{45} +\textcolor{black}{84} + \textcolor{black}{70} + \textcolor{black}{21} + \textcolor{black}{1}\\
    &= 233 = F_{13}.
\end{align*}
We leave it to the reader to easily prove that we have the following result.
\begin{theorem}
Let $k \in \mathbb{N}$. Then $\lambda = i$ satisfies $f_{4k}(\lambda) = F_{4k+1}$ where $F_n$ denotes the $n^{\mathrm{th}}$ Fibonacci number.
\end{theorem}
However, something more compelling occurs when we graph the roots of $f_n(\lambda) = F_{n+1}$. For example, in the $n=12$ case, we know that $\lambda = i$ is a root. But what about the other 11 roots? Graphing the roots of $f_{12}(\lambda) = F_{13}$ in the complex plane, we obtain the following graph:
\begin{center}
    \includegraphics[width=4in]{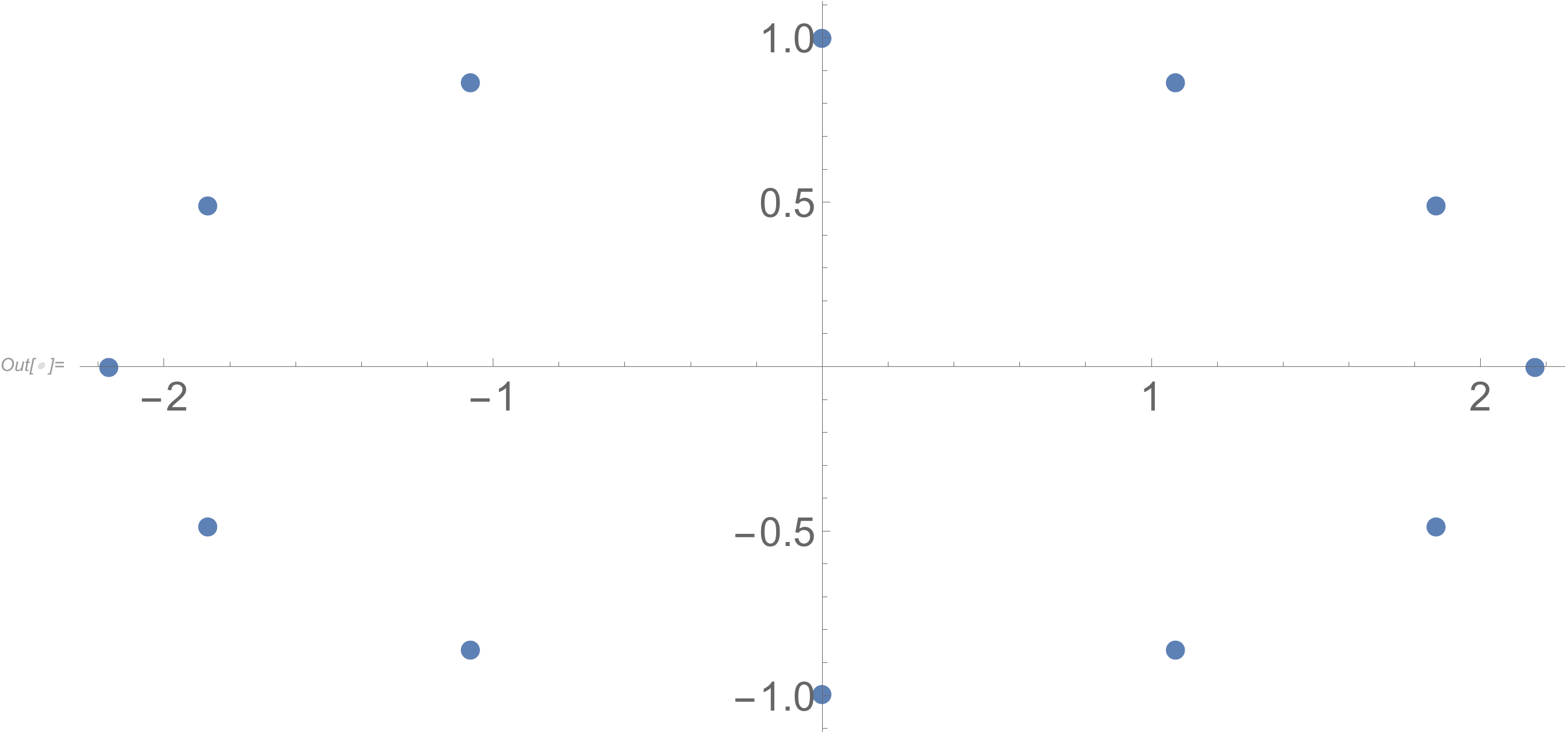}
\end{center}
These roots appear to lie on an ellipse. This is not unique though to the $f_n = F_{n+1}$ where $n$ is a multiple of 4. Consider the roots of $f_{29}(\lambda) = F_{30}$. The following is the polynomial $f_{29}(\lambda)$:
\begin{align*}
    f_{29}(\lambda) =  -\lambda ^{29} & +28 \lambda ^{27}-351 \lambda ^{25}+2600 \lambda ^{23}-12650 \lambda ^{21}+42504 \lambda ^{19}\\
    &-100947 \lambda ^{17}+170544 \lambda ^{15}-203490 \lambda ^{13}+167960 \lambda ^{11}\\
    &-92378 \lambda ^9+31824 \lambda ^7-6188 \lambda ^5+560 \lambda ^3-15 \lambda.
\end{align*}
Graphing the roots of $f_{29}(\lambda) = F_{30}$ in the complex plane, we obtain the following graph:
\begin{center}
    \includegraphics[width=4in]{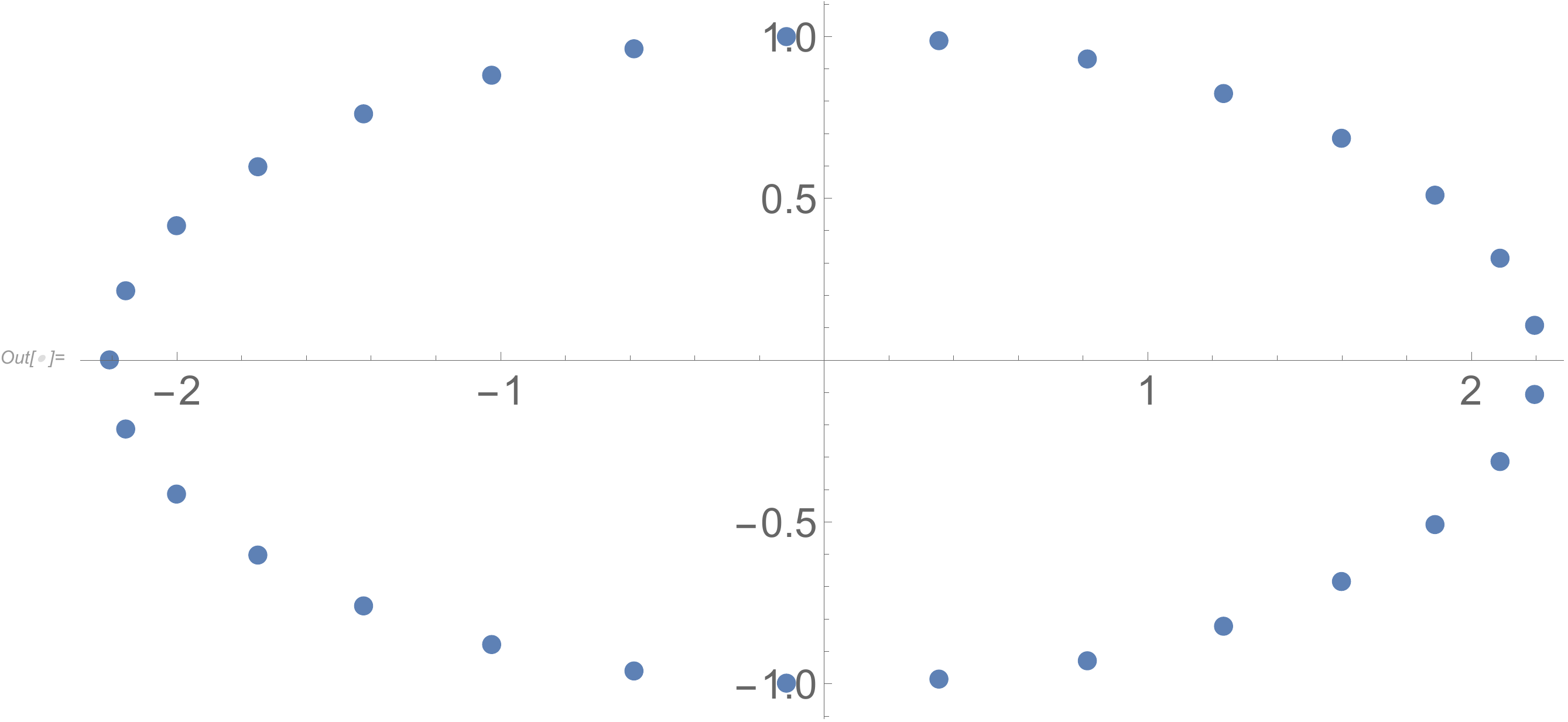}
\end{center}
What if we change a single coefficient? If we change the coefficient of $\lambda^{25}$ in $f_{29}(\lambda)$ from $-351$ to $-350$ (denoting this modified polynomial by $\widetilde{f_{29}}(\lambda)$), and we set the polynomial $\widetilde{f_{29}}(\lambda)$ equal to $F_{30}$, then we get
\begin{align*}
    F_{30} =  -\lambda ^{29} & +28 \lambda ^{27} \textcolor{red}{\mathbf{-350}} \lambda ^{25}+2600 \lambda ^{23}-12650 \lambda ^{21}+42504 \lambda ^{19}\\
    &-100947 \lambda ^{17}+170544 \lambda ^{15}-203490 \lambda ^{13}+167960 \lambda ^{11}\\
    &-92378 \lambda ^9+31824 \lambda ^7-6188 \lambda ^5+560 \lambda ^3-15 \lambda.
\end{align*}
Graphing the roots of $\widetilde{f_{29}}(\lambda) = F_{30}$ in the complex plane, we obtain the following graph:

\begin{center}
    \includegraphics[width=4in]{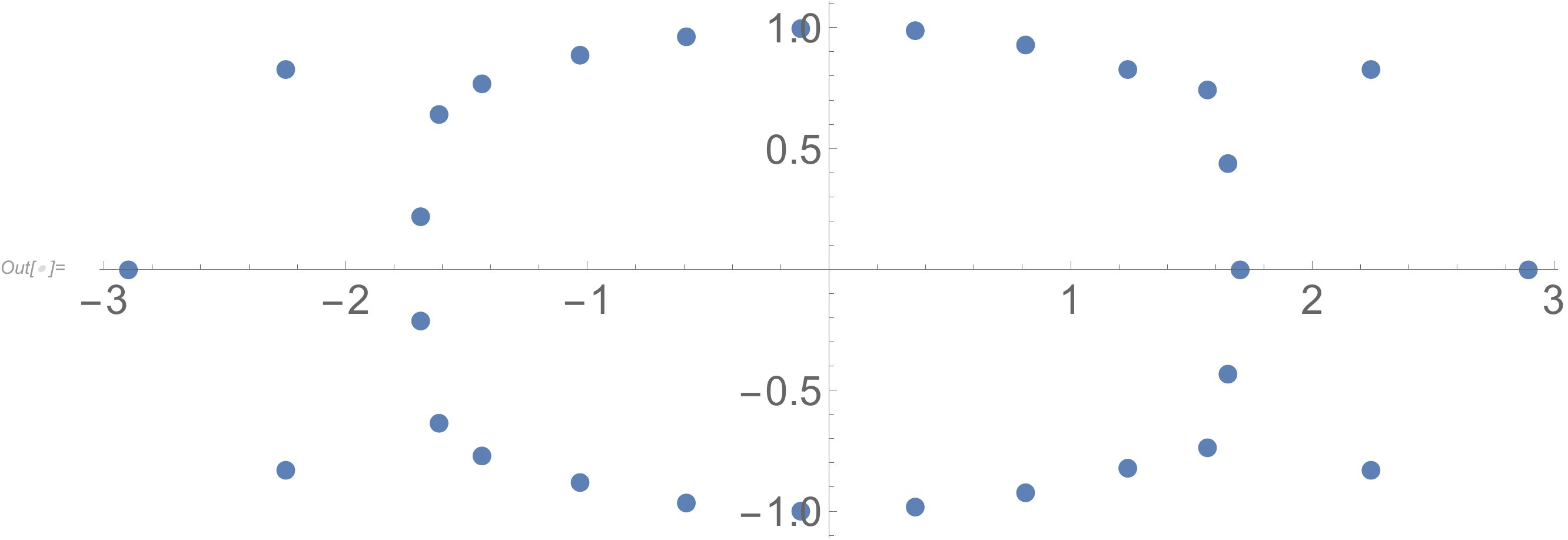}
\end{center}
Observe that we no longer get an ellipse. Clearly there is something special about the roots of $f_{n}(\lambda) = F_{n+1}$. This leads one to the following conjecture.
\begin{conjecture}\label{conj:ellipse}
The roots of $f_{n}(\lambda) = F_{n+1}$ lie on an ellipse.
\end{conjecture}
Lastly, from the two examples above of the graphs of the roots of $f_{12}(\lambda) = F_{13}$ and $f_{29}(\lambda) = F_{30}$, we observe that for the roots $\lambda = a + bi$, the imaginary parts $b$ in absolute value appear to be bounded above by 1, while the real parts $a$ in absolute value increased slightly (in absolute value) from the value $a \approx 2.16648$ in the $n=12$ case to the value $a \approx -2.20796$ in the $n=29$ case. Does this pattern continue to hold as we increase the $n$ value in $f_{n}(\lambda) = F_{n+1}$? Here are the roots of $f_{201}(\lambda) = F_{202}$ with a perfect ellipse fitting the 201 distinct roots!
\begin{center}
    \includegraphics[width=5.5in]{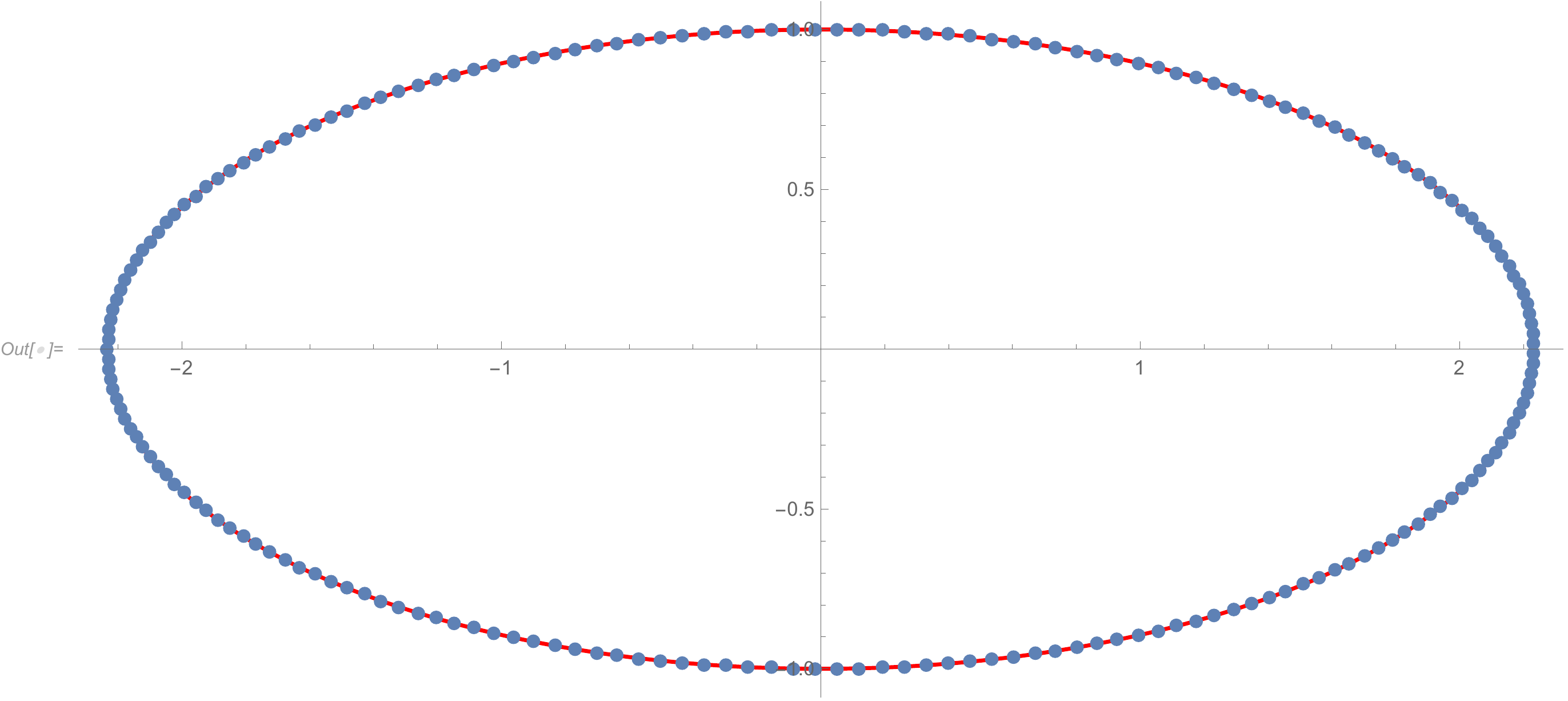}
\end{center}
In the above example, there is exactly one real root and it has value approximately $-2.23206$. Observe that this is larger in absolute value than the one real root in the $n=29$ case, which is approximately $-2.20796$. Moreover, the root with the largest imaginary part in the $n=201$ case is approximately (up to 10 decimal places) $-0.0174265601 + 0.9999693615 \, i$. Observe the imaginary part does not exceed 1 but is larger than the imaginary part in the $n=29$ case, whose root with largest imaginary part is approximately (up to 10 decimal places) $-0.118795230904 + 0.9985023199 \, i$. From further compelling evidence via \texttt{Mathematica} on a large number of concrete examples, we offer the following conjecture.
\begin{conjecture}
Consider the roots $a+bi$ of the equation $f_{n}(\lambda) = F_{n+1}$. Then the following hold:
\begin{itemize}
\item If $n$ is even, then there are two distinct real roots and $n-2$ distinct complex roots occurring in $\frac{n-2}{2}$ conjugate pairs.
\item If $n$ is odd, then there is one real root with negative parity and $n-1$ distinct complex roots occurring in $\frac{n-1}{2}$ conjugate pairs.
\item The real parts $\Re(a+bi)$ are unbounded as $n$ increases.
\item The imaginary parts $\Im(a+bi)$ are bounded above by $1$ and bounded below by $-1$, as $n$ increases.
\end{itemize}
\end{conjecture}

\section{Acknowledgments}
Among the many locations in Eau Claire, WI, where this research was done, the authors thank The Plus where they conducted research while enjoying pizza. They also thank the grassy fields of Phoenix Park where author Gullerud came up with the proof of the $\spec$ containments.



\newpage

\begin{appendices}
\renewcommand{\thesection}{A}
\section*{Appendix}

\subsection{The research team}
The research team for this project are Emily Gullerud, Rita Johnson, and Dr.~aBa Mbirika. This research was done in 2017--2018 at the University of Wisconsin-Eau Claire (UWEC). Emily is currently completing a PhD program in mathematics at the University of Minnesota. Rita is currently a software engineer in Utah. Dr.~aBa continues to happily teach mathematics and joyously conduct research at UWEC.

\begin{center}
\includegraphics[height=4in]{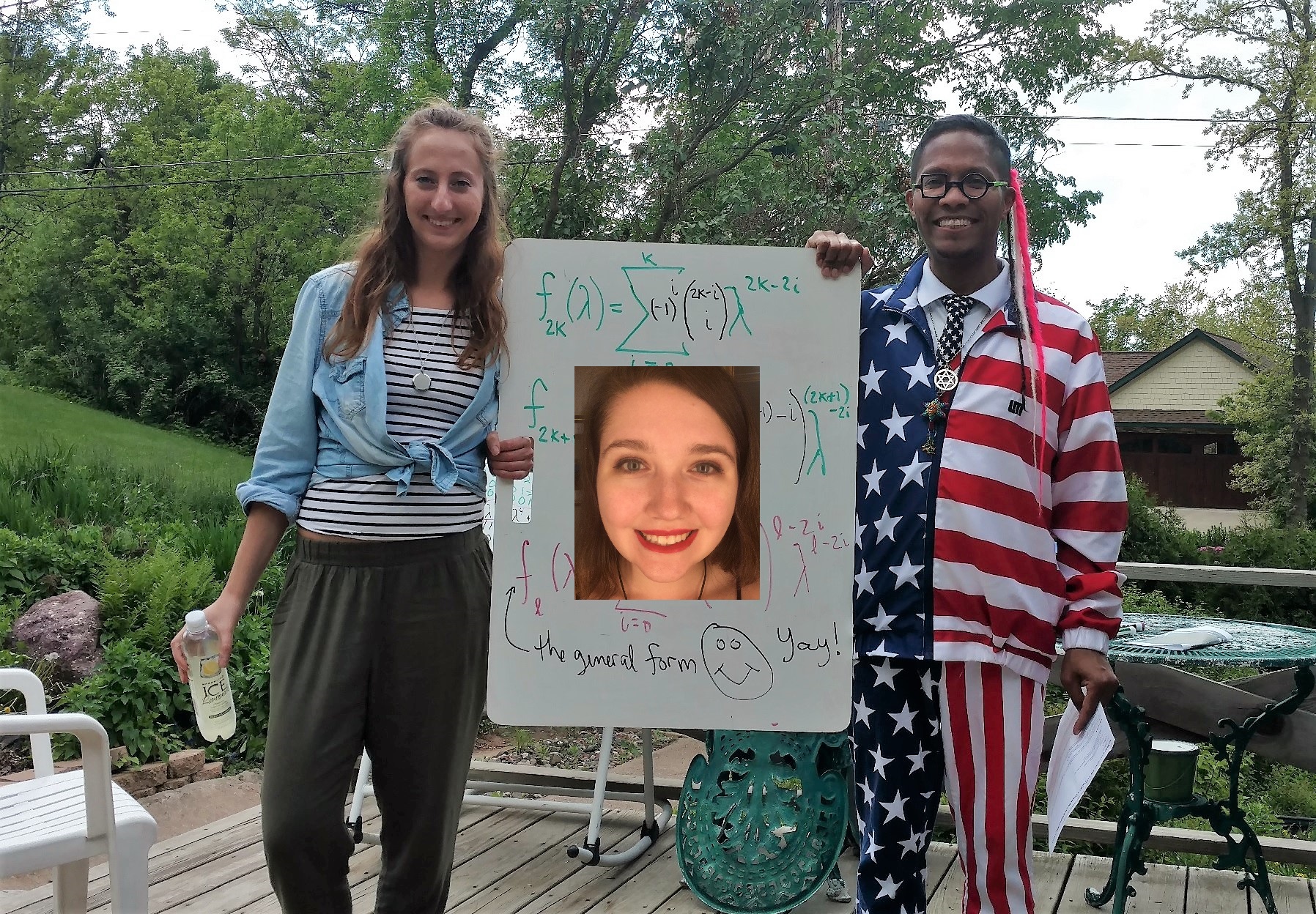}
\end{center}

\vspace{-.35in}

\begin{center}
\textcolor{blue}{\scriptsize Left to right are authors Rita Johnson, Emily Gullerud's head, and aBa Mbirika.\footnote{The authors thank the hospitality of author Rita's uncles Rick and Mike whose porch provided a stimulating environment on Memorial Day in 2017 for us to prove that Equations~\eqref{eqn:even_case} and \eqref{eqn:odd_case}, shown on the whiteboard in the picture above, satisfy the three-term recurrence relation in Theorem~\ref{thm:closed_form_satifies_the_recurrence_relation}.}}
\end{center}
\end{appendices}

\end{document}